\documentclass{article}
\usepackage[makeroom]{cancel}

\usepackage{xcolor}
\usepackage{hyperref}       % hyperlinks
 \usepackage{url}            % simple URL typesetting
\usepackage[utf8]{inputenc} % allow utf-8 input
\usepackage[accepted]{icml2021}

\usepackage{microtype}
 \usepackage{subfigure}
\usepackage{booktabs}
\usepackage{amsfonts}       % blackboard math symbols
\usepackage{nicefrac}       % compact symbols for 1/2, etc.

\usepackage{graphicx}

\usepackage{amsmath, amsthm, amssymb, mathtools, bm, acronym}
\usepackage{cuted}
\usepackage{enumitem}
\setlist[enumerate]{wide=0pt, widest=99,leftmargin=\parindent}
\setlist[itemize]{wide=0pt, widest=99,leftmargin=\parindent}
\usepackage{pgf, tikz}
\usetikzlibrary{arrows, automata}
\usetikzlibrary{chains}
\usetikzlibrary{fit}
\usetikzlibrary{matrix}
\usepackage{wrapfig}
\usepackage{pgflibraryarrows}		%optional
\usepackage{pgflibrarysnakes}		%optional

\usepackage{epsfig}
\usetikzlibrary{shapes,arrows}
\usetikzlibrary{shapes.symbols,patterns} % for source symbols
\usepackage{pgfplots}
\usetikzlibrary{pgfplots.groupplots}
\usetikzlibrary{spy}

\pgfplotsset{compat=1.3}
\usetikzlibrary{arrows.meta}
\usepgfplotslibrary{patchplots}
\usepackage{grffile}
\usepackage{amsmath}
\usepackage{comment}

\usepackage{amsthm}
\usepackage{amsmath,mathtools}
\usepackage{amssymb,amstext}
\usepackage{bbm} % for bbm numbers
\usepackage{bm}
\usepackage{dsfont}

\usepackage{todonotes}

\usepackage{footnote}

\usepackage{algorithm} 
\usepackage{algorithmic} 
\usepackage[algo2e,ruled,vlined,linesnumbered,noend]{algorithm2e}
\newtheorem{mythm}{Theorem}
\newtheorem{myprop}[mythm]{Proposition}

\newtheorem{mylem}[mythm]{Lemma}

\newtheorem{myrmk}[mythm]{Remark}

\theoremstyle{definition}
\newtheorem{definition}{Definition}

\newcommand{\Dc}[2]{\mathsf{D_c}(#1 \left\|\right. #2)}

\newcommand{\mc}{\mathcal}

\newcommand{\indic}[1]{\mathsf{1}_{#1}} %Entropy

\newcommand{\D}[2]{\mathsf{D}\!\left({#1}\right| \!\!\left|{#2}\right)} %relative entropy
\DeclareMathOperator{\st}{s.t.}

\usepackage{todonotes}

\usepackage{float}
\newcommand{\psip}{\Psi(x,P)}
\newcommand{\nuxt}{P}
\newcommand{\sxt}{S}

\newcommand{\gxt}{g}
\DeclareMathOperator*{\argmin}{argmin}
\DeclareMathOperator*{\argmax}{argmax}

\newcommand{\inprod}[2]{\ensuremath{\left\langle{#1}\vphantom{\big|},\vphantom{\big|}{#2}\right\rangle}}

\newcommand{\R}{\ensuremath{\mathbb{R}}}

\newcommand{\tr}[1]{{\rm tr}(#1)}

\newcommand{\mb}{\mathbb}
\renewcommand{\emph}{\textbf}

\def\N{\mathbb{N}}
\def\Re{\mathbb{R}}

\def\st{\mathrm{s.t.}}

\DeclareMathOperator{\cl}{cl}
\DeclareMathOperator{\interior}{int}

\icmlsetsymbol{equal}{*}

\icmlkeywords{Machine Learning, ICML}

\usepackage{calc, cleveref}

\newcommand{\drop}[1]{}

\usetikzlibrary{calc,positioning}

\begin{document}
% \captionsetup[figure]{textfont=footnotesize}
\twocolumn[
\icmltitle{Distributionally Robust Optimization with Markovian Data}

\begin{icmlauthorlist}
\icmlauthor{Mengmeng Li}{epf}
\icmlauthor{Tobias Sutter}{epf}
\icmlauthor{Daniel Kuhn}{epf}
\end{icmlauthorlist}

\icmlaffiliation{epf}{Risk Analytics and Optimization Chair, \'Ecole Polytechnique F\'ed\'erale de Lausanne}

\icmlcorrespondingauthor{Mengmeng Li}{mengmeng.li@epfl.ch}
\vskip 0.3in
]
\printAffiliationsAndNotice{}
\begin{abstract}
We study a stochastic program where the probability distribution of the uncertain problem parameters is unknown and only indirectly observed via finitely many correlated samples generated by an unknown Markov chain with $d$ states. 
We propose a data-driven distributionally robust optimization model to estimate the problem's objective function and optimal solution. By leveraging results from large deviations theory, we derive statistical guarantees on the quality of these estimators. The underlying worst-case expectation problem is nonconvex and involves $\mathcal O(d^2)$ decision variables. Thus, it cannot be solved efficiently for large~$d$. By exploiting the structure of this problem, we devise a customized Frank-Wolfe algorithm with convex direction-finding subproblems of size $\mathcal O(d)$. We prove that this algorithm finds a stationary point efficiently under mild conditions. The efficiency of the method is predicated on a dimensionality reduction enabled by a dual reformulation. 
Numerical experiments indicate that our approach has better computational and statistical properties than the state-of-the-art methods.
\end{abstract}

%%%%%%%%%%%%%%%%%
%% SEC. Introduction
%%%%%%%%%%%%%%%%%
\section{Introduction}

Decision problems under uncertainty are ubiquitous in machine learning, engineering and economics. Traditionally, such problems are modeled as stochastic programs that seek a decision $x\in X\subseteq \mb R^n$ minimizing an expected loss ${\mb E}_{\mb P}[L(x,\xi)]$, where the expectation is taken with respect to the distribution~$\mb P$ of the random problem parameter~$\xi$. However, more often than not, $\mb P$ is unknown to the decision maker and can only be observed indirectly via a finite number of training samples $\xi_1,\hdots, \xi_T$. The classical sample average approximation (SAA) replaces the unknown probability distribution $\mb P$ with the empirical distribution corresponding to the training samples and solves the resulting empirical risk minimization problem \cite{ref:Shapiro-14}. Fuelled by modern applications in machine learning, however, there has recently been a surge of alternative methods for data-driven optimization complementing SAA. Ideally, any meaningful approach to data-driven optimization should display the following desirable properties. \vspace{-12pt}
\begin{enumerate}[label = (\roman*)]\setlength\itemsep{-0.2em}
\item \label{item:optimizers:curse} \textbf{Avoiding the optimizer's curse.} It is well understood that decisions achieving a low empirical cost on the training dataset may perform poorly on a test dataset generated by the same distribution. In decision analysis, this phenomenon is called the \textit{optimizer's curse}, and it is closely related to overfitting in statistics~\cite{smith2006optimizer}. Practically useful schemes should mitigate this detrimental effect.
\item \textbf{Statistical guarantees.} \label{item:iid:assumption} Most statistical guarantees for existing approaches to data-driven optimization critically rely on the assumption that the training samples are independent and identically distributed (i.i.d.). This assumption is often not justifiable or even wrong in practice. Practically useful methods should offer guarantees that remain valid when the training samples display serial dependencies.
\item \textbf{Computational tractability.}\label{item:comp:tractability}~For a data-driven scheme to be practically useful it is indispensable that the underlying optimization problems can be solved efficiently.
\end{enumerate}\vspace{-7pt}
While the SAA method is computationally tractable, it is susceptible to the optimizer's curse if training data is scarce; see, {\em e.g.},~\cite{ref:vanParys:fromdata-17}. \textit{Distributionally robust optimization} (DRO) is an alternative approach that mitigates overfitting effects \cite{delage2010distributionally,goh2010distributionally, wiesemann2014distributionally}. DRO seeks worst-case optimal decisions that minimize the expected loss under the most adverse distribution from within a given ambiguity set, that is, a distribution family characterized by certain known properties of the unknown data-generating distribution. DRO has been studied since Scarf's seminal treatise on the ambiguity-averse newsvendor problem \cite{ref:scarf-58}, but it has gained thrust only with the advent of modern robust optimization techniques \cite{bertsimas2004price, ben2009robust}. In many cases of practical interest, DRO problems can be reformulated exactly as finite convex programs that are solvable in polynomial time. Indeed, such reformulations are available for many different ambiguity sets defined through generalized moment constraints \cite{delage2010distributionally,goh2010distributionally, wiesemann2014distributionally, ref:Kallus-18}, $\phi$-divergences \cite{ben2013robust,NIPS2016_4588e674}, Wasserstein distances \cite{Wass17Monh, ref:DROtutorial-19}, or maximum mean discrepancy distances \cite{ref:Jegelka:MMD-19,ref:Kirschner-20}. %Nowadays, DRO is an active research area in machine learning and operations research. \\  

\hspace*{5mm} With the notable exceptions of \cite{ref:Dou-19,ref:Mannor-20,ref:Sutter-19,ref:duchi-2016}, we are not aware of any data-driven DRO models for non-i.i.d.\ data.
In this paper we apply the general framework by \citet{ref:Sutter-19} to data-driven DRO models with Markovian training samples and propose an efficient algorithm for their solution. Our DRO scheme is perhaps most similar to the one studied by \citet{ref:duchi-2016}, which can also handle Markovian data. 
%where a general framework for DRO problems is presented. While the method \cite{ref:duchi-2016} can handle Markovian data it 
However, this scheme differs from ours in two fundamental ways. First, while \citet{ref:duchi-2016} work with $\phi$-divergence ambiguity sets, we use ambiguity sets inspired by a statistical optimality principle recently established by \citet{ref:vanParys:fromdata-17} and~\citet{ref:Sutter-19} using ideas from large deviations theory. Second, the statistical guarantees by \citet{ref:duchi-2016} depend on unknown constants, whereas our confidence bounds are explicit and easy to evaluate. \citet{ref:Dou-19} assume that the training samples are generated by a first-order autoregressive process, which is neither a generalization nor a special case of a finite-state Markov chain. %These data-generating processes are complementary and neither is a special case of the other. 
Indeed, while \citet{ref:Dou-19} can handle continuous state spaces, our model does not assume a linear dependence on the previous state. \citet{ref:Mannor-20} investigate a finite-state Markov decision process (MDP) with an unknown transition kernel, and they develop a DRO approach using a Wasserstein ambiguity set for estimating its value function. While MDPs are more general than the static optimization problems considered here, the out-of sample guarantees in~\citep[Theorem~4.1]{ref:Mannor-20} rely on the availability of several i.i.d.\ sample trajectories. In contrast, our statistical guarantees require only one single trajectory of (correlated) training samples. We also remark that MDP models can be addressed with online mirror descent methods~\cite{jin2020efficiently}. Unfortunately, it is doubtful whether such methods could be applied to our DRO problems because the statistically optimal ambiguity sets considered in this paper are nonconvex.

\hspace*{5mm} In summary, while many existing data-driven DRO models are tractable and can mitigate the optimizer's curse, there are hardly any statistical performance guarantees that apply when the training samples fail to be i.i.d.\ and when there is only one trajectory of correlated training samples. This paper addresses this gap. Specifically, we study data-driven decision problems where the training samples are generated by a time-homogeneous, ergodic Markov chain with~$d$ states. \citet{ref:Sutter-19} show that statistically optimal data-driven decisions for such problems are obtained by solving a DRO model with a conditional relative entropy  ambiguity set. The underlying worst-case expectation problem is nonconvex and involves~$d^2$ decision variables. To our best knowledge, as of now there exist no efficient algorithms for solving this hard optimization problem. We highlight the following main contributions of this paper.\vspace{-12pt}
\begin{itemize}\setlength\itemsep{-0.2em}
\item We apply the general framework by \citet{ref:Sutter-19}, which uses ideas from large deviations theory to construct statistically optimal data-driven DRO models, to decision problems where the training data is generated by a time-homogeneous, ergodic finite-state Markov chain. We prove that the resulting DRO models are asymptotically consistent.
\item We develop a customized Frank-Wolfe algorithm for solving the underlying worst-case expectation problems in an efficient manner. This is achieved via the following steps.\vspace{-2pt}
\begin{enumerate} [label = (\roman*)]\setlength\itemsep{-0em}
    \item We first reparametrize the problems to move the nonconvexities from the feasible set to the objective function.
    \item We then develop a Frank-Wolfe algorithm for solving the reparametrized nonconvex problem of size~$\mathcal O(d^2)$.
    \item Using a duality argument, we show that the direction-finding subproblems with a linearized objective function are equivalent to convex programs of size~$\mathcal O(d)$ with a rectangular feasible set that can be solved highly efficiently.
\end{enumerate}\vspace{-2pt}
\item We prove that the proposed Frank-Wolfe algorithm converges to a stationary point of the nonconvex worst-case expectation problem at a rate $\mc O(1/\sqrt{M})$, where $M$ denotes the number of iterations. Each iteration involves the solution of a convex $d$-dimensional minimization problem with a smooth objective function and a rectangular feasible set. The solution of this minimization problem could be further accelerated by decomposing it into $d$ one-dimensional optimization problems that can be processed in parallel.

\item We propose a Neyman-Pearson-type hypothesis test for determining the Markov chain that generated a given sequence of training samples, and we construct examples of Markov chains that are arbitrarily difficult to distinguish on the basis of a finite training dataset alone.

\item In the context of a revenue maximization problem where customers display a Markovian brand switching behavior, we show that the proposed DRO models for Markovian data outperform classical DRO models, which are either designed for i.i.d.\ data or based on different ambiguity sets. 
\end{itemize}\vspace{-12pt}
\paragraph*{Notation.}
The inner product of two vectors $a, b \in \mathbb{R}^{m}$ is denoted by $\langle a, b\rangle=a^{\top} b$, and the $n$-dimensional probability simplex is defined as $\Delta_n = \{x \in \mb R^n_+ : \sum_{i=1}^n x_i = 1\}$. The relative entropy between two probability vectors $p,q\in\Delta_n$ is defined as $\mathsf{D}(p \| q)=\sum_{i=1}^n p_i \log \left(p_i/q_i\right)$, where we use the conventions $0\log(0/q)=0$ for $q\geq 0$ and $p\log(p/0)=\infty$ for $p>0$. The closure and the interior of a subset $\mathcal D$ of a topological space are denoted by $\cl \mathcal D$ and $\mathsf{int}\mathcal D$, respectively. For any $n\in\mathbb N$ we set $[n]=\{1,\hdots,n\}$. %The support of~$A\in\mb R^{n\times n}$ is defined as $\mathsf{supp}(A)=\{(i,j)\in [n]^2 : A_{ij}\neq 0\}$, and we 
We use~$A_{i\cdot}$ and~$A_{\cdot i}$ to denote the $i$-th row and column of a matrix~$A\in\mb R^{m\times n}$, respectively. All proofs are relegated to the Appendix.

\section{Data-driven DRO with Markovian Data}\label{sec:problem:statement}
In this paper we study the single-stage stochastic program
\begin{equation} \label{eq:SP}
\min_{x\in X} \mathbb{E}_{\mathbb{P}}[L(x,\xi)],
\end{equation}
where $X\subseteq \mb R^n$ is compact, $\xi$ denotes a random variable valued in $\Xi=\{1,2,\hdots,d \}$, and $L:X\times \Xi\to\mb R$ represents a known loss function that is continuous in~$x$. We assume that all random objects ({\em e.g.}, $\xi$) are defined on a probability space~$(\Omega, \mc F, \mb P)$ whose probability measure $\mb P$ is unknown but belongs to a known ambiguity set $\mc A$ to be described below. We further assume that $\mb P$ can only be observed indirectly through samples $\{\xi_t\}_{t\in\mb N}$ from a time-homogeneous ergodic Markov chain with state space $\Xi$ and deterministic initial state~$\xi_0\in\Xi$. We finally assume that the distribution of~$\xi$ coincides with the unique stationary distribution of the Markov chain. In the following we let $\theta^\star\in\mb R^{d\times d}_{++}$ be the unknown stationary probability mass function of the doublet $(\xi_t,\xi_{t+1})\in\Xi^2$, that is, we set
\begin{equation}\label{eq:def:doublet:frequency}
\theta^\star_{ij}=\lim_{t\to\infty}\mb P (\xi_t=i, \xi_{t+1}=j) \quad \forall i,j\in\Xi.
\end{equation}
By construction, the row sums of $\theta^\star$ must coincide with the respective column sums, and thus $\theta^\star$ must belong to
\begin{align*}
  \Theta\!=\!\Bigg\{\!\theta\in\mb R^{d\times d}_{++}:\! \sum_{i,j=1}^d \theta_{ij}=1, \sum_{j=1}^d\theta_{ij} = \sum_{j=1}^d \theta_{ji}~\forall i\in\Xi\!\Bigg\},
\end{align*}
that is, the set of all strictly positive doublet probability mass functions with balanced marginals. Below we call elements of $\Theta$ simply {\em models}. Note that any model $\theta\in\Theta$ encodes an ergodic time-homogeneous Markov chain with a unique row vector $\pi_\theta\in\mb R^{1\times d}_{++}$ of stationary probabilities and a unique transition probability matrix $P_{\theta}\in\mb R_{++}^{d\times d}$ defined through $(\pi_{\theta})_i=\sum_{j\in\Xi}\theta_{ij}$ and $(P_{\theta})_{ij}=\theta_{ij}/(\pi_\theta)_i$, respectively. Hence, the elements of $\pi_\theta$ sum to~$1$, $P_\theta$ represents a row-stochastic matrix, and $\pi_\theta=\pi_\theta P_\theta$. 
% \begin{align*}
%   (P_{\theta})_{ij}&= \mb P_{\theta}	(\xi_{t+1}=j|\xi_t=i) = \frac{\mb P_{\theta}(\xi_t=i, \xi_{t+1}=j)}{\mb P_{\theta}(\xi_t=i)} \\&= \frac{{\theta}_{ij}}{\sum_{j\in\Xi}{\theta}_{ij}} \quad \forall i,j\in\Xi.
% \end{align*}
%For every element $\theta'\in\Theta'$, we can define the stationary distribution vector $(\pi_{\theta'})_i=\sum_{j\in\Xi}\theta'_{ij}$. If $(\pi_{\theta'})_i=\sum_{j\in\Xi}{\theta'}_{ij}=0$, then we may define without loss of generality $(P_{\theta})_{ij}=1$ if $j=i$ and $(P_{\theta'})_{ij}=0$ otherwise.
Moreover, any $\theta\in\Theta$ induces a probability measure~$\mb P_\theta$ on $(\Omega,\mc F)$ with
\begin{equation*}
  \mb P_\theta(\xi_t= i_t\;\forall t=1,\ldots, T ) = (P_\theta)_{\xi_0 i_{1}}\prod_{t=1}^{T-1} (P_\theta)_{i_t i_{t+1}}
\end{equation*}
for all $(i_1,\hdots,i_T)\in \Xi^T$ and $T\in\mb N$. We are now ready to define the ambiguity set as $\mc A=\{\mb P_{\theta}:\theta \in \Theta\}$. Below we use $\mb E_\theta$ to denote the expectation operator with respect to $\mb P_\theta$. By construction, there exists a model $\theta^\star\in\Theta$ corresponding to the unknown true probability measure~$\mb P$ such that $\mb P_{\theta^\star}=\mb P$. Estimating $\mb P$ is thus equivalent to estimating the unknown true parameter~$\theta^\star$. Given a finite training dataset $\{\xi_t\}_{t=1}^T$, a natural estimator for $\theta^\star$ is the empirical doublet distribution $\widehat\theta_T\in\mb R^{d\times d}_+$ defined through
\begin{equation} \label{MC:estimator}
	(\widehat\theta_T)_{ij} = \frac{1}{T} \sum_{t=1}^{T} \indic{(\xi_{t-1},\xi_{t})=(i,j)}
\end{equation}
for all $i,j\in\Xi$. The ergodic theorem ensures that $\widehat \theta_T$ converges $\mb P_\theta$-almost surely to $\theta$ for any model $\theta\in \Theta$ \citep[Theorem~4.1]{ross2010introduction}. 
In the remainder we define $\Theta'=\Delta_{d\times d}$ as the state space of~$\widehat\theta_T$, which is strictly larger than~$\Theta$. For ease of notation, we also define the {\em model-based predictor}
% \subsection{Alternate description of Markov chains} \label{subsec:Markov:chain:doublet:frequency}
% We consider an alternate description of Markov chains that turns out to be helpful in our problem setting.
% Suppose the stochastic process $\{\xi_t\}_{t\in\mb N}$ forms a time-homogeneous, ergodic Markov chain defined on $\Xi=\{1,\dots, d\}$ with known initial state $\sigma\in\Xi$ and the matrix $\theta^\star$ denotes the unknown stationary probability mass function of the doublet $(\xi_t,\xi_{t+1})$.
$c(x,\theta)= \mb E_\theta[L(x,\xi)]=\sum_{i,j=1}^dL(x,i)\theta_{ij}$ as the expected loss of~$x$ under model~$\theta$. Note that $c(x,\theta)$ is jointly continuous in~$x$ and~$\theta$ due to the continuity of $L(x,\xi)$ in~$x$. 

% We define a probability mass function $\theta$ on $\Xi^2$ by
% \begin{equation}\label{eq:def:doublet:frequency}
% \theta_{ij}=\lim_{t\to\infty}\mb P_\theta(\xi_t=i, \xi_{t+1}=j) \quad \forall i,j\in\Xi.
% \end{equation}

%{\color{blue} The speed of this convergence can be characterized via the underlying large deviation principle. Moreover, under the assumption that the initial state of the Markov chain is given, the estimator~\eqref{MC:estimator} is known to be a sufficient statistic \cite[p.~14]{ref:Billingsley-61}, indicating that by ``compressing" the available data into the estimator $\widehat\theta_T$, no information about the underlying model $\theta$ is lost.}

%\subsection{Distributionally robust estimation}
As we will show below, it is natural to measure the discrepancy between a model $\theta\in\Theta$ and an estimator realization $\theta'\in\Theta'$ by the conditional relative entropy. 
\begin{definition}[Conditional relative entropy] 
\label{def:conditional_relative_entropy}
The conditional relative entropy of $\theta'\in\Theta'$ with respect to $\theta\in \Theta$ is
\begin{align*}
	&\Dc{\theta'}{\theta} =\sum_{i=1}^d (\pi_{\theta'})_i \, \D{(P_{\theta'})_{i\cdot}}{(P_\theta)_{i\cdot}} \\&~=\sum_{i,j=1}^d \theta'_{ij} \left( \log\left( \frac{\theta'_{ij}}{\sum_{k=1}^d \theta'_{ik}} \right) -  \log\left( \frac{\theta_{ij}}{\sum_{k=1}^d \theta_{ik}}\right) \right),
\end{align*}
where the invariant distributions $\pi_{\theta},\pi_{\theta'}\in \mb R^{1\times d}_+$ and the transition probability matrices $P_{\theta},P_{\theta'}\in\mb R^{d\times d}_+$ corresponding to~$\theta$ and~$\theta'$, respectively, are defined in the usual way.
\end{definition}
Note that if the training samples $\{\xi_t\}_{t\in\mb N}$ are i.i.d., then the conditional relative entropy of~$\theta'$ with respect to~$\theta$ collapses to the relative entropy of~$\pi_{\theta'}$ with respect to~$\pi_\theta$. Using the conditional relative entropy, we can now define the \textit{distributionally robust predictor} $\widehat c_r:X\times\Theta'\to\mb R$ through
\begin{subequations}\label{eq:DRO:general}
\begin{equation} \label{def:DRO}
\widehat c_r(x,\theta') = \max\limits_{\theta\in\cl \Theta} \left\{ c(x,\theta) \ : \ \Dc{\theta'}{\theta}\leqslant r\right\} 
\end{equation}
if~\eqref{def:DRO} is feasible and $\widehat c_r(x,\theta') = \max_{\theta\in\cl \Theta} c(x,\theta)$ otherwise. Note that $\widehat c_r(x, \theta')$ represents the worst-case expected cost of~$x$ with respect to all probability measures in the ambiguity set~$\{\mb P_\theta:\Dc{\theta'}{\theta}\leqslant r\}$, which can be viewed as an image (in~$\mc A$) of a conditional relative entropy ball (in~$\Theta$) of radius~$r\geqslant0$ around $\theta'$. We know from \citep[Proposition~5.1]{ref:Sutter-19} that the sublevel sets of $\Dc{\theta'}{\theta}$ are compact for all fixed~$\theta'$, and thus the maximum in \eqref{def:DRO} is attained whenever the problem is feasible. % Proposition~\ref{prop:properties:conditional:entropy} in the appendix shows that $\Dc{\theta'}{\theta}$ is lower semi-continuous in $\theta$. 
The distributionally robust predictor~\eqref{def:DRO} also induces a {\em distributionally robust prescriptor} $\widehat x_r:\Theta'\to X$ defined through
\begin{equation} \label{def:DRO:decision}
\widehat x_r(\theta') \in \arg\min_{x\in X} \widehat c_r(x,\theta')
\end{equation}
\end{subequations}
Embracing the distributionally robust approach outlined above, if we only have access to $T$ training samples, we will implement the data-driven decision $\widehat x_r(\widehat\theta_T)$ and predict the corresponding expected cost as~$\widehat c_r( \widehat x_r(\widehat \theta_T) , \widehat \theta_T)$, which we will refer to, by slight abuse of terminology, as the {\em in-sample risk}. %where $\widehat \theta_T$ denotes the empirical doublet distribution \eqref{MC:estimator}. 
To assess the quality of the data-driven decision~$\widehat x_r(\widehat\theta_T)$ we use two performance measures that stand in direct competition with each other: the \textit{out-of-sample risk} $c(\widehat x_r(\widehat\theta_T),\theta)$ and the \textit{out-of-sample disappointment}
\begin{align*}
    &\mb P_\theta\left( c(\widehat x_r(\widehat\theta_T),\theta) > \widehat c_r (\widehat x_r(\widehat\theta_T),\widehat\theta_T)\right)
\end{align*}
with respect to any fixed model~$\theta$. Intuitively, the out-of-sample risk represents the true expected loss of~$\widehat x_r(\widehat\theta_T)$, and the out-of-sample disappointment quantifies the probability that the in-sample risk (the predicted loss) strictly underestimates the out-of-sample risk (the actual loss) under $\mb P_\theta$. Following \citet{ref:Sutter-19}, we will consider a data-driven decision as desirable if it has a {\em low} out-of-sample risk and a {\em low} out-of-sample disappointment under the unknown true model~$\theta^\star$. This is reasonable because the objective of the original problem~\eqref{sec:problem:statement} is to {\em minimize} expected loss. Optimistically underestimating the loss could incentivize na\"ive decisions that overfit to the training data and therefore lead to disappointment in out-of-sample tests. Thus, it makes sense to keep the out-of-sample disappointment small. 

%such that the corresponding out-of-sample disappointment \eqref{eq:out:of:sample:disappointment} is small. We will show that this is the case for data-driven DRO decision defined in \eqref{def:DRO:decision}, which distinct our DRO approach from classical SAA approaches.

%%%%%%%%%%%%%
\section{Statistical Guarantees}\label{ssec:statistical:guarantees}
The distributionally robust optimization model~\eqref{eq:DRO:general} with a conditional relative entropy ambiguity set has not received much attention in the extant literature and may thus seem exotic at first sight. As we will show below, however, this model offers---in a precise sense---{\em optimal} statistical guarantees if the training samples are generated by an (unknown) time-homogeneous ergodic Markov chain. The reason for this is that the empirical doublet distribution~\eqref{MC:estimator} is a sufficient statistic for~$\theta$ and satisfies a large deviation principle with rate function~$\Dc{\theta'}{\theta}$ under~$\mb P_\theta$ for any~$\theta\in\Theta$.

%stating that the conditional relative entropy is the natural extension of the relative entropy in the i.i.d.~setting to address correlated data described by an underlying Markov chain. 

Before formalizing these statistical guarantees, we first need to recall some basic concepts from large deviation theory. We refer to the excellent textbooks by \citet{dembo2009large} and  \citet{hollander2008largedeviationis} for more details.
\begin{definition}[{Rate function \citep[Section~2.1]{dembo2009large}}]
\label{def:rate_function:original}
A function $I:\Theta'\times \Theta\rightarrow [0,\infty]$ is called a rate function if $I(\theta',\theta)$ is lower semi-continuous in $\theta'$. % for every $\theta\in\Theta$ is termed a rate function.
\end{definition}

\begin{definition}[Large deviation principle] \label{def:wldt}
An estimator~$\widehat\theta_T\in \Theta'$ satisfies a large deviation principle (LDP) with rate function $I$ if for all $\theta\in\Theta$ and Borel sets $\mc D \subseteq \Theta'$ we have
\begin{subequations}
\label{eq:ldp_exponential_rates:old}
\begin{align}
	\label{eq:ldp_exponential_rates_lb:old}
	-\inf_{\theta' \in \mathsf{int} {\mc D}} \, I(\theta', \theta)~& \leq \liminf_{T\to \infty}~\frac1T \log \mb P_{\theta}\left( \widehat \theta_T \in \mc D \right) \\
	& \leq \limsup_{T\to \infty}~\frac1T \log \mb P_\theta\left( \widehat \theta_T \in \mc D \right) 
	\\&\leq -\inf_{\theta' \in \cl  {\mc D}} \,  I(\theta', \theta).
	\label{eq:ldp_exponential_rates_ub:old}
\end{align}
\end{subequations}
\end{definition}
Conveniently, the empirical doublet distribution \eqref{MC:estimator} obeys an LDP with the conditional relative entropy as rate function.

\begin{mylem}[LDP for Markov chains {\citep[Theorem 3.1.13]{dembo2009large}}]\label{Lem:LDP}
The empirical doublet distribution~\eqref{MC:estimator} satisfies an LDP with rate function $\Dc{\theta'}{\theta}$.
\end{mylem}
%Strictly speaking \eqref{MC:estimator} satisfies an LDP with rate function $I(\theta',\theta) = \Dc{\theta'}{\theta}$ under the assumption that the Markov chain has a strictly positive transition matrix \cite[Theorem~3.1.13]{dembo2009large}. For the theorem to hold under the weaker irreducibility assumption, one needs to replace the state space $\Xi^2$ of the doublet Markov chain by $\{(i,j)\in\Xi^2 \ : \ (P_\theta)_{ij}>0\}$ \cite[p.~79]{dembo2009large}.
Lemma~\ref{Lem:LDP} is the key ingredient to establish the following out-of-sample guarantees for the distributionally robust predictor and the corresponding prescriptor.
\begin{mythm}[{Out-of-sample guarantees \citep[Theorems~3.1 \&~3.2]{ref:Sutter-19}}]\label{thm:DRO:guarantees}
The distributionally robust predictor and prescriptor defined in \eqref{eq:DRO:general} offer the following guarantees.\vspace{-10pt}
\begin{enumerate}[label = (\roman*)]\setlength\itemsep{-0.2em}
\item \label{item:DRO:guarantees:predictor} For all $\theta\in\Theta$, $x\in X$, we have\vspace{-7pt}
\begin{align*}
    \limsup_{T\to\infty} \frac{1}{T} \log \mb P_\theta \left( c(x,\theta) > \widehat c_r(x,\widehat\theta_T)\right)\leqslant -r.
\end{align*}\vspace{-7pt}
\item \label{item:DRO:guarantees:prescriptor} For all $\theta\in\Theta$, we have\vspace{-7pt}
\begin{align*}
\limsup_{T\to\infty} \frac{1}{T} \log \mb P_\theta \left( c(\widehat x_r(\widehat\theta_T),\theta) > \widehat c_r(\widehat x_r(\widehat\theta_T),\widehat\theta_T)\right)\leqslant -r.\end{align*}
\end{enumerate}
% \vspace{-15pt}
\end{mythm}
Theorem~\ref{thm:DRO:guarantees} asserts that the out-of-sample disappointment of the distributionally robust predictor and the corresponding prescriptor decay at least as fast as $e^{-rT+o(T)}$ asymptotically as the sample size~$T$ grows, where the decay rate~$r$ coincides with the radius of the conditional relative entropy ball in~\eqref{eq:DRO:general}. %are indeed ``good" in the sense that its corresponding out-of-sample disappointments~\eqref{eq:out:of:sample:disappointment} decay exponentially at a rate $r$ that can be controlled by the decision maker.
Note also that~$\widehat x_r(\widehat\theta_T)$ can be viewed as an instance of a data-driven prescriptor, that is, any Borel measurable function that maps the training samples $\{\xi_t\}_{t=1}^T$ to a decision in~$X$. One can then prove that~$\widehat x_r(\widehat\theta_T)$ offers the best (least possible) out-of-sample risk among the vast class of data-driven prescriptors whose out-of-sample disappointment decays at a prescribed rate of at least~$r$. Maybe surprisingly, this optimality property holds uniformly across all models~$\theta\in\Theta$ and thus in particular for the unknown true model~$\theta^\star$; see \citep[Theorem~3.2]{ref:Sutter-19}. It is primarily this Pareto dominance property that makes the distributionally robust optimization model~\eqref{eq:DRO:general} interesting.

Next, we prove that if the radius $r$ of the conditional relative entropy ball tends to $0$ at a sufficiently slow speed as~$T$ increases, then the distributionally robust predictor converges $\mb P_\theta$-almost surely to the true expected loss of the decision~$x$. %, and the distributionally robust prescriptor~\eqref{def:DRO:decision} converges $\mb P_\theta$-almost surely to a minimizer of a variant of the stochastic program~\eqref{eq:SP} with $\mb P_\theta$ replacing $\mb P$.

%evaluated at the estimator $c_{r_T}(x,\widehat\theta_T)$, converges to the optimal value of the original problem $c(x,\theta)$ almost surely, and also $c_{r_T}(\widehat x_{r_T}(\widehat\theta_T),\widehat\theta_T)$ converges to $\min_{x\in X} c(x,\theta)$ almost surely.

\begin{mythm}[Strong asymptotic consistency] \label{thm:asymptotic:consistency}
If $r_T\geq \frac{d}{T}$ for all $T\in\N$ and $\lim_{T\to\infty} r_T = 0$, then %the DRO estimators defined in \eqref{eq:DRO:general} are (strongly) consistent, i.e.,
% \begin{enumerate}[label = (\roman*)]
    % \label{item:consistency:predictor} 
\[
    \lim_{T\to\infty} \! \widehat c_{r_T}(x,\widehat\theta_T) = c(x,\theta)\quad\mb P_\theta\text{-a.s.}\quad\forall x\in X,\; \theta\in\Theta.
\]
%     $$
%     \lim_{T\to\infty} \widehat c_{r_T}(x,\widehat\theta_T) = c(x,\theta) \quad \mb P_\theta  \text{-a.s.};
% $$
% \item \label{item:consistency:prescriptor} $\displaystyle \lim_{T\to\infty} \! \widehat c_{r_T}(\widehat{x}_{r_T}(\widehat{\theta}_T),\widehat\theta_T) = \min_{x\in X} c(x,\theta)$ $\mb P_\theta  \text{-a.s.}$ $\forall \theta\in\Theta$.
% $$
%     \lim_{T\to\infty} \widehat c_{r_T}(\widehat{x}_{r_T}(\widehat{\theta}_T),\widehat\theta_T) = \min_{x\in X} c(x,\theta) \quad \mb P_\theta  \text{-a.s.}
% $$
% \end{enumerate}
\end{mythm}

\section{Numerical Solution}\label{ssec:approximation:algorithm}

Even though the distributionally robust optimization model~\eqref{eq:DRO:general} offers powerful statistical guarantees, it is practically useless unless the underlying optimization problems can be solved efficiently. In this section we develop a numerical procedure to compute the predictor~\eqref{def:DRO} for a fixed~$x$ and~$\theta'$. 
This is challenging for the following reasons. First, the conditional relative entropy~$\Dc{\theta'}{\theta}$ is nonconvex in~$\theta$ (see Remark~\ref{rmk:nonconvexity} in Appendix~\ref{app:sec:markov:chains}), and thus problem~\eqref{def:DRO} has a nonconvex feasible set. In addition, problem~\eqref{def:DRO} involves $d^2$ decision variables (the entries of the matrix~$\theta$) and thus  its dimension is high already for moderate values of~$d$.
% A natural approach, given the convexity of problem \eqref{def:DRO} is to solve its dual program, which (hopefully) is computationally easier. Unfortunately, a direct dualization of \eqref{def:DRO} leads to a dual program that is difficult to solve, as its dual function requires the solution of a non-linear system of equations, see Appendix~\ref{ref:app:ex:dual} for more details. 
% \\  
% \hspace*{5mm} 

In the following we reformulate~\eqref{def:DRO} as an equivalent optimization problem with a convex feasible set and a nonconvex objective function, which is amenable to efficient numerical solution via a customized Frank-Wolfe algorithm \cite{frankwolfe, pmlr-v28-jaggi13}. To this end, recall from Definition~\ref{def:conditional_relative_entropy} that the conditional relative entropy $\Dc{\theta'}{\theta}$ can be expressed as a weighted sum of relative entropies between corresponding row vectors of the transition probability matrices $P_{\theta '}$ and $P_{\theta}$. By construction, $P_{\theta}$ is a row-stochastic matrix, and the stationary distribution~$\pi_\theta$ corresponding to~$P_\theta$ is a normalized non-negative row vector with $\pi_\theta P_\theta =\pi_\theta$. Routine calculations show that
\begin{equation} \label{eq:pi:as:function:of:P}
	(\pi_\theta)^\top = (A_d(P_\theta))^{-1} \begin{bmatrix}
0&0& \hdots & 1
\end{bmatrix}^\top,
\end{equation}
where the matrix $A_d(P_\theta)\in\mb R^{d\times d}$ is defined through
\begin{equation*} %\label{eq:def:ad:matrix}
	(A_d(P_\theta))_{ij}=\begin{cases}
(P_\theta)_{ii}-1 & \text{for} \ i=j, i<d \\
(P_\theta)_{ji} & \text{for} \ i\ne j, i<d \\
1 & \text{for} \ i=d.
\end{cases}
\end{equation*}  
From now on we denote by ${\mc P}\subseteq \Re^{d\times d}$ the set of all transition probability matrices of ergodic 
Markov chains on~$\Xi$ with strictly positive entries. Applying the variable substitution $P_\theta\leftarrow\theta$ and using~\eqref{eq:pi:as:function:of:P}, we can now reformulate the optimization problem~\eqref{def:DRO} in terms of~$P_\theta$.
\begin{mylem}[Reformulation]\label{reparam}
	If $\theta'>0$, then the worst-case expectation problem~\eqref{def:DRO} is equivalent to
	\begin{equation} \label{eq:lipschitz:DRO:problem}
	\widehat c_r(x,\theta')= \max_{P\in {\mc D}_r(\theta')} \Psi(x,P),
	\end{equation}
	where $\Psi(x,P)=\sum_{i=1}^d L(x,i) (A_d(P)^{-1})_{id}$ and
	\[
	    \textstyle \mc D_r(\theta') = \left\{ P \in{\mc P}: \sum_{i=1}^d (\pi_{\theta'})_i\D{(P_{\theta^\prime})_{i\cdot}}{(P)_{i\cdot}}\leqslant r\right\}.
	\]
	%\vspace{-10pt}
\end{mylem}

In the following we will sometimes abbreviate~$\mc D_r(\theta')$ by $\mc D$ when the dependence on~$\theta'$ and~$r$ is irrelevant. Problem~\eqref{eq:lipschitz:DRO:problem} is more attractive than~\eqref{def:DRO} from a computational point of view because its feasible set is closely related to the popular relative entropy uncertainty sets, which have received considerable attention in robust optimization \cite{ben2013robust}. On the other hand, while the objective function~$c(x,\theta)$ of the original problem~\eqref{def:DRO} was linear in~$\theta$, the objective function $\psip$ of \eqref{eq:lipschitz:DRO:problem} is nonconvex in~$P$. These properties suggest that~\eqref{eq:lipschitz:DRO:problem} can be solved efficiently and exactly if its objective function is replaced with a linear approximation. We thus develop a customized Frank-Wolfe algorithm that generates approximate solutions for problem~\eqref{eq:lipschitz:DRO:problem} by solving a sequence of linearized oracle subproblems. While our main Frank-Wolfe routine is presented in Algorithm~\ref{alg:FW}, its key subroutine for solving the oracle subproblems is described in Algorithm~\ref{alg:subprob}.

\begin{algorithm}[tb]
      \caption{Frank-Wolfe algorithm for solving~\eqref{eq:lipschitz:DRO:problem}}
      \label{alg:FW}
\begin{algorithmic}[1]
  \STATE {\bfseries Input:} $\ r>0,\ \theta'>0, \ x\in X, \ \varepsilon>0,\ \pi_{\theta'},\ P_{\theta^\prime}$
  \STATE {\bfseries Output:} $P$ 
  \STATE Initialize  $\nuxt^{(0)}, \gxt^{(0)} = 2\varepsilon, m=0$
  \WHILE{$\gxt^{(m)} > \varepsilon$}
  \STATE Compute $\displaystyle\sxt^{(m)} \in \argmax_{S \in {\mc D}} \langle S, \nabla_{P} \Psi(x,\nuxt^{(m)})\rangle$  \label{line:linearized:subprob}
  \STATE Compute $\gxt^{(m)} =  \inprod{\sxt^{(m)} - \nuxt^{(m)}}{\nabla_{P} \Psi(x,\nuxt^{(m)})}$
  \IF{$ \gxt^{(m)} \leqslant \varepsilon$}
  \STATE Return $\nuxt=\nuxt^{(m)}$ \ENDIF
  \STATE Compute $\!\displaystyle\gamma_m \!\!\in \!\!\argmax_{\gamma \in [0,1]}  \! \Psi(x,(\nuxt^{(m)} \!+ \!\gamma (\sxt^{(m)} \!- \!\nuxt^{(m)})))$ \label{line:linesearch}
  \STATE $\nuxt^{(m+1)} = \nuxt^{(m)} + \gamma_m (\sxt^{(m)} - \nuxt^{(m)})$
  \STATE $m \rightarrow m+1$
  \ENDWHILE
\end{algorithmic}
\end{algorithm}
Algorithm~\ref{alg:FW} generates a sequence of approximate solutions for problem \eqref{eq:lipschitz:DRO:problem} that enjoy rigorous optimality guarantees. %This point is made precise in the following theorem.

\begin{mythm}[Convergence of Algorithm~\ref{alg:FW}] \label{thm:global:convergence:FW}
For any fixed $\varepsilon>0$ and $\theta'>0$, Algorithm~\ref{alg:FW} finds an approximate stationary point for problem~\eqref{eq:lipschitz:DRO:problem} with a Frank-Wolfe gap of at most $\varepsilon$ in $\mc O(1/\varepsilon^2)$ iterations.
% \vspace{-10pt}
\end{mythm}
The proof of Theorem~\ref{thm:global:convergence:FW} leverages convergence results for generic Frank-Wolfe algorithms by~\citet{lacoste2016convergence}.
% ; finally, by exploiting the underlying problem structure we show that any local maximum is a global maximum, which is nonconvex since the problem~\eqref{eq:lipschitz:DRO:problem} is nonconvex. 
A detailed proof is relegated to Appendix~\ref{app:sec:theoretical:guarantees}. \\  
\hspace*{0mm} \tikzstyle{block} = [rectangle, draw, fill=gray!20,
    text width=16em, text centered, rounded corners, minimum height=3em]
\tikzstyle{line} = [draw, very thick, color=black!50, -latex']
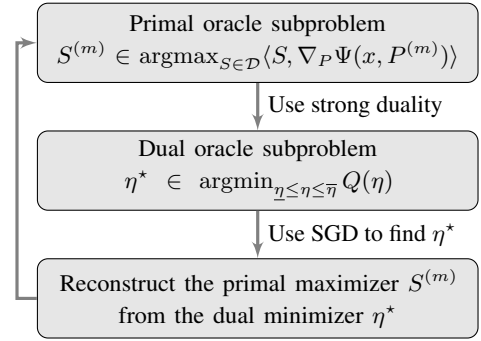
\begin{figure}
\centering
\begin{tikzpicture}[scale=0.64, node distance = 1.7cm, auto]
    \node [block] (init) {\footnotesize Primal oracle subproblem \\$\sxt^{(m)} \in \argmax_{S \in {\mc D}} \langle S, \nabla_{P} \Psi(x,\nuxt^{(m)})\rangle$};
    \node [block, below of=init] (dual) {\footnotesize  Dual oracle subproblem\\ $\eta^{\star}\in\argmin_{\underline{\eta}\le\eta\le\overline\eta} Q(\eta)$};
    \node [block, below of=dual] (sgd) {\footnotesize  Reconstruct the primal maximizer $\sxt^{(m)}$\\ from the dual minimizer $\eta^{\star}$};
    \path [line] (init) --node [, color=black]  {\footnotesize  Use strong duality}(dual);
    \path [line] (dual) -- node [, color=black] {\footnotesize  Use SGD to find $\eta^{\star}$} (sgd);
    \path [line] (sgd) --++  (-5,0) |- (init);
\end{tikzpicture}
\caption{\small Flow diagram of Algorithm~\ref{alg:subprob}.} 
\label{fig:motifAlg2}\vspace{-2pt}
\end{figure}
We now explain the implementation of Algorithm~\ref{alg:FW}.
Note that the only nontrivial step is the solution of the oracle subproblem in line~$5$. To construct this subproblem we need the gradient $\nabla_P \Psi(x,P)$, which can be computed in closed form. 
% by leveraging~\citep[Theorem 4.1]{caswell2013sensitivity}. 
Full details are provided in Appendix~\ref{app:sec:theoretical:guarantees}. Even though the oracle subproblem is convex by construction (see Proposition~\ref{prop:D:convex}), it is still computationally expensive because it involves $d^2$ decision variables $P\in\mathbb R^{d\times d}$ subject to $\mathcal O(d)$ linear constraints as well as a single conditional relative entropy constraint. Algorithm~\ref{alg:subprob} thus solves the {\em dual} oracle subproblem, which involves only $d$ decision variables $\eta\in\mathbb R^d$ subject to box constraints. Strong duality holds because~$P_{\theta'}$ represents a Slater point for the primal problem (see Proposition~\ref{prop:linear:predictor}). The dual oracle subproblem is amenable to efficient numerical solution via Stochastic Gradient Descent (SGD), and a primal maximizer can be recovered from any dual minimizer by using the problem's first order optimality conditions (see Proposition~\ref{prop:linear:predictor}). The high-level structure of Algorithm~\ref{alg:subprob} is visualized in Figure~\ref{fig:motifAlg2}.
%whose main structure is visualized by Figure~\ref{fig:motifAlg2}. We first invoke duality theory of convex optimization, since the linearized subproblem can be shown to be convex, see Proposition~\ref{prop:D:convex}. Moreover, we suggest solving the corresponding dual program via Stochastic Gradient Descent (SGD). We then retrieve the primal optimizer via a first order condition (see Theorem~\ref{thm:linear:predictor} for a formal justification). % These key steps of Algorithm~\ref{alg:subprob} for solving the linearized subproblem are highlighted by the flow-diagram in Figure~\ref{fig:motifAlg2}.

To construct the dual oracle subproblem, fix a decision~$x\in X$ and an anchor point~$\nuxt^{(m)}\in\mc P$, and set~$C = \nabla_{P} \Psi(x,\nuxt^{(m)})$. In addition, define $\underline{\eta}_{i} = \max_j\{C_{ij}\}$ and
\begin{equation*}
\overline{\eta}_i \!=\! \frac{1}{1- e^{-r}}\!\left(\!d \max _{i, j} \{C_{i j}\}\! - e^{-r}\tr {C^\top P_{\theta'}}\right) - \sum_{k\ne i} \underline{\eta}_{k}.
\end{equation*}
The dual subproblem minimizes $Q(\eta) =  \sum_{i=1}^d Q_i(\eta)$ over the box~$[\underline \eta, \overline \eta]$, where $Q_i(\eta)=\eta_i-\lambda^\star(\eta)/d$ %for all $i=1,\hdots, d$ 
and
\begin{equation*}
    \lambda^\star(\eta)= \exp \left(\sum_{i,j=1}^{d} (\pi_{\theta'})_i (P_{\theta'})_{ij} \log \left(\frac{\eta_i-C_{ij}}{(\pi_{\theta'})_i }\right)-r\right)
\end{equation*}
is a convex function. As~$Q(\eta)$ is reminiscent of an empirical average, the dual oracle subproblem lends itself to efficient numerical solution via SGD, which obviates costly high-dimensional gradient calculations. % whose closed form can be found in \citep[Theorem 4.1]{caswell2013sensitivity}.
In addition, the partial derivatives of $Q_i(\eta)$ are readily available in closed form as
\begin{equation*}
\frac{\partial}{\partial \eta_i} Q_i(\eta^{(n)})=d-(\pi_{\theta'})_i \sum_{j=1}^d \frac{(P_{\theta^\prime})_{ij}}{\eta^{(n)}_i-C_{ij}}.
\end{equation*}
We set the SGD learning rate to~$\ell =K N^{-1/2}$ for some constant~$K>0$, which ensures that the suboptimality of the iterates generated by the SGD algorithm decays as $\mc O(1/\sqrt{N})$ after $N$ iterations; see~\citep[Section~2.2]{nemirovski09Robust}. A summary of Algorithm~\ref{alg:subprob} in pseudocode is given below.

\begin{algorithm}[H]
 \caption{Solution of the oracle subproblem %Computing linearized subproblem: subroutine of line~\ref{line:linearized:subprob} in Algorithm~\ref{alg:FW}
 }\label{alg:subprob}
\begin{algorithmic}[1]
  \STATE {\bfseries Input:} $ \underline{\eta}, \ \overline{\eta}, \ \ell, P_{\theta'},\  N$
  \STATE {\bfseries Output:} $\sxt^{(N)} \in \arg\min_{S \in {\mc D}} \inprod{S}{-\nabla_{P} \Psi(x,\nuxt^{(m)})}$
  \STATE Initialize $\eta^{(0)}$ with $\underline{\eta}_{i} \leqslant \eta^{(0)}_i \leqslant \overline{\eta}_i,   i=1,\ldots,d$
  \FOR{$n=0,\ldots , N$}
  \FOR{$i=1,\ldots, d$}
  \STATE $y^{(n)}_i=\eta^{(n)}_i-\ell \frac{\partial}{\partial \eta_i} Q_i(\eta^{(n)})$ 
  \ENDFOR
  \STATE $\eta^{(n+1)}\in\argmin_{\underline{\eta}_{i} \leqslant s_i \leqslant \overline{\eta}_i, i=1,\ldots,d}\|s-y^{(n)}\|^2_2$ 
  \ENDFOR
  \STATE $\eta^\star=\eta^{N+1}$, $\lambda^\star=\lambda^\star(\eta^\star)$
  \STATE $(\sxt^{(N)})_{d(i-1)+j}= (\pi_{\theta'})_i (P_{\theta^\prime})_{ij} \lambda^\star/(\eta^\star_i-\tilde{c}^{(i)}_{j})$ for $i,j=1,\ldots,d$
\end{algorithmic}
\end{algorithm}
\begin{mythm}[{Convergence of  Algorithm~\ref{alg:subprob}}]\label{thm:correct:algsub}
	For any fixed $\varepsilon>0$, Algorithm~\ref{alg:subprob} outputs an $\varepsilon$-suboptimal solution of the primal oracle subproblem
	%$\argmax_{S \in {\mc D}} \langle S,\nabla_{P} \Psi(x,\nuxt^{(m)})\rangle$ 
	in $\mc O(1/\varepsilon^2)$ iterations. 
% 	\vspace{-7pt}
\end{mythm}
%Other than the linearized subproblem, the rest of Algorithm~\ref{alg:FW} is routine computation of the Frank-Wolfe method~\cite{frankwolfe}.
Theorem~\ref{thm:correct:algsub} follows from standard convergence results in~\citep[Section~2.2]{nemirovski09Robust} applied to the dual oracle subproblem. We finally remark that the stepsize~$\gamma_m$ to be computed in line~\ref{line:linesearch} of Algorithm~\ref{alg:FW} can be obtained via a direct line search method.

\begin{myrmk}[Overall complexity of Algorithm~\ref{alg:FW}]
Denote by $\varepsilon_1$ the error of of the Frank-Wolfe algorithm (Algorithm~\ref{alg:FW}) provided by Theorem~\ref{thm:global:convergence:FW}, and by $\varepsilon_2$ the error of the subroutine (Algorithm~\ref{alg:subprob}) provided by Theorem~\ref{thm:correct:algsub}. That is, suppose we solve the Frank-Wolfe subproblem approximately in the sense that $ \langle S^{(m)}, \nabla_{P} \Psi(x, P^{(m)})\rangle\ge \max_{S \in \mathcal{D}}\langle S, \nabla_{P} \Psi(x, P^{(m)})\rangle-\kappa \varepsilon_2,$ for some constant $\kappa>0$. We denote by $C_f$ the curvature constant of the function $\Psi$. By \cite{lacoste2016convergence}, after running $\mathcal O((1+2\kappa\varepsilon_2/C_f)/\varepsilon_1^2)$ many iterations of Algorithm~\ref{alg:FW} we achieve an accumulated error below $\varepsilon_1$. As Algorithm~\ref{alg:FW} calls Algorithm~\ref{alg:subprob} in total $\mathcal O(1/\varepsilon_1^2)$ times, the overall complexity coincides with the number of operations needed to solve $\mathcal O(1/\varepsilon_1^2)$ times a $d$-dimensional minimization problems of a smooth convex function over a compact box.
This significantly reduces the computational cost that would be needed to solve in each step the original $d^2$-dimensional nonconvex problem~\eqref{def:DRO} directly. A further reduction of complexity is possible by applying a randomized Frank-Wolfe algorithm, as suggested by~\citet{reddi2016stochastic}.
\end{myrmk}

%\subsection{The Prescriptor Problem}
Given an efficient computational method to solve the predictor problem~\eqref{def:DRO}, we now address the solution of the prescriptor problem~\eqref{def:DRO:decision}.
Since our Frank-Wolfe algorithm only enables us to access the values of the objective function~$\widehat{c}_{r}(x, \theta^{\prime})$ for fixed values of~$x\in X$ and~$\theta'\in\Theta'$, we cannot use gradient-based algorithms to optimize over~$x$. Nevertheless, we can resort to derivative-free optimization (DFO) approaches, which can find local or global minima under appropriate regularity conditions on~$\widehat{c}_{r}(x, \theta^{\prime})$.

Besides popular heuristic algorithms such as simulated annealing~\cite{metropolis1953equation} or genetic algorithms~\cite{holland1992adaptation}, there are two major classes of DFO methods: direct search and model-based methods. 
The direct search methods use only comparisons between function evaluations to generate new candidate points, while the model-based methods construct a smooth surrogate objective function. 
We focus here on direct search methods due to their solid worst-case convergence guarantees. Descriptions of these algorithms and convergence proofs can be found in the recent textbook by~\citet{audet2017derivative}. For a survey of model-based DFO see~\cite{conn2009introduction}. Off-the-shelf software to solve constrained DFO problems includes \texttt{fminsearch} in Matlab or \texttt{NOMAD} in C++; see~\cite{rios2013derivative} for an extensive list of solvers.

When applied to the extreme barrier function formulation of the prescriptor problem~\eqref{def:DRO:decision}, which is obtained by ignoring the constraints and adding a high penalty to each infeasible decision~\cite{audet2006mesh}, the directional direct-search method by~\citet{vicente2013worst} offers an attractive worst-case convergence guarantee thanks to~\citep[Theorem 2 \& Corollary 1]{vicente2013worst}. The exact algorithm that we use to solve~\eqref{def:DRO:decision} is outlined in Algorithm~\ref{alg:prescri} in Appendix~\ref{sec:dfo-algorithm}.
\begin{mythm}[Convergence of Algorithm~\ref{alg:prescri}~\cite{vicente2013worst}]
	Fix any $\theta'\in\Theta'$ and $\varepsilon>0$. If $ \nabla_x \widehat{c}_{r}(x,\theta') $ is Lipschitz continuous in $x$, then Algorithm~\ref{alg:prescri} outputs a solution~$x_N$ with $\|\nabla_x \widehat{c}_{r}(x_N,\theta^{\prime})\|\le \varepsilon$ after $N\sim \mathcal{O}( \varepsilon^{-2})$ iterations. \\If $ \widehat{c}_{r}(x,\theta') $ is also convex in~$x$, then we further have
\begin{align*}\vspace{-10pt}
    \|\widehat{c}_{r}(x_N,\theta^{\prime})-\min_{x\in X} \widehat{c}_{r}(x,\theta^{\prime})\|\le \varepsilon.
\end{align*}% \vspace{-10pt}
\end{mythm}

\section{An Optimal Hypothesis Test}
\label{sect:hypo:test}
We now develop a Neyman-Pearson-type hypothesis test for determining which one of two prescribed Markov chains has generated a given sequence~$\xi_1,\hdots,\xi_T$ of training samples. Specifically, 
%$We provide an interpretation showing the difficulty of identifying the correct data-generating Markov model based on finitely many observations, via hypothesis testing. Assume 
we assume that these training samples were either generated by the Markov chain corresponding to the doublet probability mass function~$\theta^{(1)}\in\Theta$ or by the one corresponding to~$\theta^{(2)}\in\Theta$, that is, we have $\theta\in\{\theta^{(1)}, \theta^{(2)}\}$. Next, we construct a decision rule encoded by the open set 
\begin{equation}\label{eq:HT:set:A}
    \mathcal{B} = \{ \theta'\in\Theta' : \Dc{\theta'}{\theta^{(1)}}<\Dc{\theta'}{\theta^{(2)}} \},
\end{equation}
which predicts~$\theta=\theta^{(1)}$ if $\widehat\theta_T\in\mathcal{B}$ and predicts~$\theta=\theta^{(2)}$ otherwise.
%Specifically, we  such that we guess the data has been generated by $\theta^{(1)}$ if $\widehat\theta_T\in\mathcal{B}$ and by $\theta^{(2)}$ otherwise. This leads to the discussion on criteria of a good decision rule $\mathcal{B}$ and magnitude of the corresponding error probabilities. 
The quality of any such decision rule is conveniently measured by the type I and type II error probabilities
\begin{align}\label{eq:HT:error:probabilities}
    \alpha_T = \mathbb{P}_{\theta^{(1)}}(\widehat\theta_T\notin \mathcal{B}) \quad \text{and}\quad
    \beta_T = \mathbb{P}_{\theta^{(2)}}(\widehat\theta_T\in \mathcal{B})
\end{align}
for~$T\in\mathbb N$, respectively. Specifically, $\alpha_T$ represents the probability that the proposed decision rule wrongly predicts~$\theta=\theta^{(2)}$ if the training samples were generated by~$\theta^{(1)}$, and $\beta_T$ represents the probability that the decision rule wrongly predicts~$\theta=\theta^{(1)}$ if the data was generated by~$\theta^{(2)}$. 

\begin{myprop}[Decay of the error probabilities] \label{eq:prop:HT}
% Consider a decision rule described by the set
% \begin{equation}\label{eq:HT:set:A}
%     \mathcal{B} = \{ \theta'\in\Theta' : \Dc{\theta'}{\theta^{(1)}}<\Dc{\theta'}{\theta^{(2)}} \}.
% \end{equation}
If $\theta^\star \in \arg\min_{\theta'\not\in\mathcal B} \Dc{\theta'}{\theta^{(1)}}$ and $r=\Dc{\theta^\star}{\theta^{(1)}}>0$, then
\begin{align*}
    &\lim_{T\to\infty}\frac{1}{T} \log \alpha_T = -r \quad \text{and} \quad
    \lim_{T\to\infty}\frac{1}{T} \log \beta_T = -r.
\end{align*}
\end{myprop}
Proposition~\ref{eq:prop:HT} ensures that the proposed decision rule is powerful in the sense that both error probabilities decay exponentially with the sample size~$T$ at some rate~$r$ that can at least principally be computed.
The next theorem is inspired by the celebrated Neyman-Pearson lemma \citep[Theorem~11.7]{ref:Cover-06} and shows that this decision rule is actually optimal in a precise statistical sense.

\begin{mythm}[Optimality of the hypothesis test] \label{thm:optimality:HT}
Let $\bar\alpha_T$ and $\bar\beta_T$ represent the type~I and type~II error probabilities of the decision rule obtained by replacing~$\mathcal B$ with any other open set~$\bar{\mathcal{B}}\subseteq\Theta'$. %be the decision rule given by \eqref{eq:HT:set:A} with associated error probabilities $\alpha_T$ and $\beta_T$ defined in \eqref{eq:HT:error:probabilities}. Let $\bar{\mathcal{B}}\subseteq\Theta'$ be any other decision rule with non-empty interior and with associated probabilities of error $\bar\alpha_T$ and $\bar\beta_T$. 
If $\lim_{T\to\infty}\frac{1}{T}\log\bar\alpha_T \leq \lim_{T\to\infty}\frac{1}{T}\log\alpha_T$, then $\lim_{T\to\infty}\frac{1}{T}\log\bar\beta_T\geq \lim_{T\to\infty}\frac{1}{T}\log\beta_T$.
\end{mythm}

Theorem~\ref{thm:optimality:HT} implies that the proposed hypothesis test is Pareto optimal in the following sense. If any test with the same test statistic~$\widehat\theta_T$ but a different set~$\bar{\mathcal B}$ has a faster decaying type~I error probability, then it must necessarily have a slower decaying type~II error probability. In Appendix~\ref{section:MCoin:example} we construct two different Markov chains that are almost indistinguishable on the basis of a finite training dataset. Indeed, we will show that the (optimal) decay rates of the type~I and type~II error probabilities can be arbitrarily small.

%************************************

\section{Revenue Maximization under a Markovian Brand Switching Model}

We now test the performance of our approach in the context of a revenue maximization model, which aims to recognize and exploit repeat-buying and brand-switching behavior. This model assumes that the probability of a customer buying a particular brand depends on the brand purchased last~\cite{herniter1961customer, chintagunta1991InvestigatingHeterogeneityBrand, gensler2007EvaluatingChannelPerformance,  leeflang2015ModelingMarkets}. Adopting the perspective of a medium-sized retailer, we then formulate an optimization problem that seeks to maximize profit from sales by anticipating the long-term average demand of each brand based on the brand transition probabilities between different customer segments characterized by distinct demographic attributes. For instance, younger people might prefer relatively new brands, while senior people might prefer more established brands. Once such insights are available, the retailer selects the brands to put on offer with the goal to maximize long-run average revenue. We assume that the brand purchasing behavior of the customers is exogenous, which is realistic unless the retailer is a monopolist or has at least significant market power. Therefore, the retailer's ordering decisions have no impact on the behavior of the customers. The ergodicity of the brand choice dynamics of each customer segment can be justified by the ergodicity of demographic characteristics~\cite{ezzati1974ForecastingMarketShares}.

\subsection{Problem Formulation}
%When modeling the above problem, the goal is to obtain an appropriate form for the revenue maximization.  The revenue from selling a particular type of goods is expressed as the product of the seles quantity of each brand times their respective prices. 
Assume that there are $d$ different brands of a particular good, and denote by $a\in\Re^d$ the vector of retail prices per unit of the good for each brand. The seller needs to decide which brands to put on offer. This decision is encoded by a binary vector~$x\in\{0,1\}^d$ with~$x_j=1$ if brand~$j$ is offered and~$x_j=0$ otherwise. To quantify the seller's revenue, we need to specify the demand for each brand. To this end, assume that the customers are clustered into~$n$ groups with similar demographic characteristics~\cite{kuo2002integration}. The percentage of customers in group~$k$ is denoted by~$w_k\in\R_+$, and the aggregate brand preferences~$\xi^{k}_t$ of the customers in group~$k$ and period~$t$ represent an ergodic Markov chain on $\Xi=\{1,\ldots,d\}$ with stationary distribution~$\pi_{\theta^{k}}$. Thus, the long-run average cost per customer and time period is
\begin{equation*}%\label{eq:ex:cost:fct}
\begin{aligned} 
c(x,\theta)&=\lim _{T \rightarrow \infty} -\frac{1}{T}\Bigg(\sum_{t=0}^{T-1} \sum_{j=1}^{d}a_j x_j \sum_{k=1}^n w_k \indic{\xi_t^{k}=j}\Bigg) \\
&=\sum_{k=1}^{n} w_k \sum_{j=1}^{d} -a_j x_{j} \left(\pi_{\theta^{k}}\right)_{j}=\sum_{k=1}^{n} w_k c^k(x,\theta^k),
%\\&=-\sum_{k=1}^{n} w_k \mathbb{E}_{\theta^{k}}[ \, a^\top D x \, ],
\end{aligned}
\end{equation*}
where $c^k(x,\theta^k)= \sum_{j=1}^{d} -a_j x_{j} (\pi_{\theta^{k}})_{j}$ and $\theta=(\theta^k)_{k=1}^n$. Here, the second equality exploits the ergodic theorem \citep[Theorem~4.1]{ross2010introduction}. It is easy to show that~$c^k(x,\theta^k)$ can be expressed as an expected loss with respect to a probability measure encoded by~$\theta^k$, much like the objective function of problem~\eqref{eq:SP}. 
%$D=\operatorname{diag}(\indic{\xi^{(k)}=1},\ldots,\indic{\xi^{(k)}=d}).$ The cost function \eqref{eq:ex:cost:fct} indicates that the problem considered is indeed an instance of \eqref{eq:SP}.
Finally, we define~$X=\{x\in\{0,1\}^d:Cx\leq b\}$ for some $C\in\R^{m\times d}$ and $b\in\R^m$. The linear inequalities~$Cx\leq b$ may capture budget constraints or brand exclusivity restrictions etc.
% We also consider a constraint regarding the available budget. Let $C\in\R^{m\times d}$ be a cost matrix, where its entries $C_{ij}$ denote the $i$-th dimensional cost for purchasing brand $j$. For instance, a three-dimensional ($m=3$) cost would involve unit price, transportation cost and inventory cost.
% Then the decision $x$, i.e., the purchasing portfolio, must satisfy $	C x \le b,$
% where $b$ is the budget vector. This modeling approach provides a trade-off for the most popular brand, since for example in the chocolate industry, the price differentiation would not be that significant as the brand influence on consumer choice, while popularity of brand may significantly influence the purchasing unit price for wholesalers.
The cost minimization problem $\min_{x\in X} c(x, \theta)$ can thus be viewed as an instance of~\eqref{eq:SP}.
% \begin{equation*}
% \begin{aligned}
% 	c\left(x^{\star}(\theta), \theta\right)&=\min_{x\in X} c(x, \theta)
% %	&=\!\!\min _{x\in\{0,1\}^d}\!\! -\Big\{\sum_{k=1}^{n} w_k \mathbb{E}_{\theta^{(k)}}[ \, a^\top D x\, ]  :  C x \le b \Big\},
% \end{aligned}
% \end{equation*}
% which clearly is an instance of \eqref{eq:SP}. 
If $\theta$ is unknown and only $T$ training samples are available, we construct the empirical doublet distributions~$\widehat \theta_T^k$ for all~$k=1,\ldots, n$ and solve the DRO problem 
\begin{align}\label{eq:ex:Markov:DRO}
	%&\widehat{c}_{r}(\widehat{x}_{r}(\widehat{\theta}_{T}), \widehat{\theta}_{T}) =
	\min _{x\in X}~
	\sum_{k=1}^n w_k\widehat c_r^{\,k}(x,\widehat\theta^{\,k}_T),
\end{align}
where
% 	\vspace{-3mm}
\begin{align}\label{eq:ex:Markov:sub}
&\widehat c_r^{\,k}(x,\widehat\theta^{\,k}_T) =\max _{\theta^{k} \in \cl\Theta}\Big\{ c^k(x,\theta^k) : \Dc{\widehat{\theta}_{T}^{\,k}}{ \theta^{k}} \leq r\Big\}.
\end{align}
Below we denote by~$\widehat c_T(x,\widehat \theta_T)$ the objective function and by~$\widehat{x}_{r}(\widehat{\theta}_{T})$ a minimizer of~\eqref{eq:ex:Markov:DRO}, where $\widehat \theta_T=(\widehat\theta_T^k)_{k=1}^n$.

\subsection{Numerical Experiments}
We now validate the out-of-sample guarantees of the proposed Markov DRO approach experimentally when the observable data is indeed generated by ergodic Markov chains and compare our method against three state-of-the-art approaches to data-driven decision making from the literature. The first baseline is the Wasserstein DRO approach by~\citet{ref:Mannor-20}, which replaces the worst-case expectation problem~\eqref{eq:ex:Markov:sub} for customer group $k$ with 
\begin{equation*}
  \widehat c_r^{\,k}(x,\widehat\theta^{\,k}_T) = \max_{P\in \mc D_r(\widehat{\theta}_{T}^{\,k})} \Psi (x,P),  
\end{equation*}
where the reparametrized objective function $\Psi(x,P)$ is defined as in Lemma~\ref{reparam}, and the ambiguity set is defined as
\[
    \textstyle \mc D_r(\widehat{\theta}_{T}^{\,k}) = \left\{ P \in{\mc P}: \mathsf{d_W}((P_{\widehat{\theta}_{T}^{\,k}})_{i\cdot},P_{i\cdot})\leqslant r \ \forall i\right\}.
\]
Here, $\mathsf{d_W}$ denotes the $1$-Wasserstein distance, where the transportation cost between two states~$i,j\in\Xi$ is set to~$|i-j|$. We also compare our method against the classical SAA approach~\cite{ref:Shapiro-14} and the DRO approach by~\citet{ref:vanParys:fromdata-17}, which replaces the conditional relative entropy in the ambiguity set~\eqref{eq:ex:Markov:sub} with the ordinary relative entropy ({\em i.e.}, the Kullback-Leibler divergence). We stress that both of these approaches were originally designed for serially independent training samples.
For small~$T$ it is likely that $\widehat \theta_T\not>0$, in which case Algorithm~\ref{alg:FW} may not converge; see Theorem~\ref{thm:global:convergence:FW}. In this case, we slightly perturb and renormalize~$\widehat\theta_T$ to ensure that~$\widehat \theta_T>0$ and~$\widehat \theta_T\in\Theta'$.

\textbf{Synthetic data.} In the first experiment we solve a random instance of the revenue maximization problem with~$n=5$ customer groups and~$d=10$ brands, where the weight vector~$w$ and the price vector~$a$ are sampled from the uniform distributions on~$\Delta_n$ and~$\{1,\ldots,10\}^d$, respectively. To construct the transition probability matrix of the Markov chain reflecting the brand switching behavior of any group~$k$, we first sample a random matrix from the uniform distribution on~$[0,1]^{d\times d}$, increase two random elements of this matrix to~$4$ and~$5$, respectively, and normalize all rows to render the matrix stochastic. As the data-generating Markov chains in this synthetic experiment are known, the true out-of-sample risk of any data-driven decision can be computed exactly.\footnote{The Matlab code for reproducing all results is available from \url{https://github.com/mkvdro/DRO_Markov}.} 
\begin{figure}[h!]
    \centering
    \scalebox{0.55}{\input{outperf_w}}   
    \caption{\small Out-of-sample risk %$c(\widehat x_r(\widehat{\theta}_T), \theta_{\text{true}})$ 
    of different data-driven decisions based on synthetic data. Solid lines and shaded areas represent empirical averages and empirical 90\% confidence intervals, respectively, computed using $100$ independent training datasets.}
    \label{fig:synth:out}\vspace{10pt}
     \centering
    \scalebox{0.55}{% This file was created by matlab2tikz.
\begin{tikzpicture}

\begin{groupplot}[group style = {group size = 3 by 1, horizontal sep = 30pt}, width=1.5in,height=1.15in]

\nextgroupplot[
title={$T=10$},
scale only axis,
xmode=log,
xmin=0.0001,
xmax=10,
xminorticks=true,
ymin=-0.05,
ymax=1.1,
xlabel={$r$},
axis background/.style={fill=white},
axis x line*=bottom,
axis y line*=left,
legend style = { column sep = 10pt, legend columns = -1, legend to name = grouplegend,}
]
\addplot [smooth,color=blue, line width=2.0pt]
  table[row sep=crcr]{%
0.0001	0.07\\
0.000359381366380463	0.07\\
0.00129154966501488	0.07\\
0.00464158883361278	0.07\\
0.0166810053720006	0.07\\
0.0599484250318941	0.07\\
0.215443469003188	0.07\\
0.774263682681127	0.07\\
2.78255940220713	0.05\\
10	0.07\\
};\addlegendentry{Markov DRO}

\addplot [smooth,color=cyan, dashed,line width=2.0pt]
  table[row sep=crcr]{%
0.0001	0.9\\
0.000359381366380463	0.9\\
0.00129154966501488	0.9\\
0.00464158883361278	0.86\\
0.0166810053720006	0.82\\
0.0599484250318941	0.9\\
0.215443469003188	0.46\\
0.774263682681127	0.0667\\        
2.78255940220713	0.0333\\
10	0.0667\\
};\addlegendentry{Wasserstein}

\addplot [color=red, dotted,line width=2.0pt]
  table[row sep=crcr]{%
0.0001	1\\
0.000359381366380463	1\\
0.00129154966501488	1\\
0.00464158883361278	1\\
0.0166810053720006	1\\
0.0599484250318941	1\\
0.215443469003188	1\\
0.774263682681127	0.38\\
2.78255940220713	0\\
10	0\\
};\addlegendentry{i.i.d.~DRO}

\addplot [smooth,color=orange, line width=2.0pt]
  table[row sep=crcr]{%
0.0001	0.96\\
0.000359381366380463	0.96\\
0.00129154966501488	0.96\\
0.00464158883361278	0.96\\
0.0166810053720006	0.96\\
0.0599484250318941	0.96\\
0.215443469003188	0.96\\
0.774263682681127	0.96\\
2.78255940220713	0.96\\
10	0.96\\
};\addlegendentry{SAA}

\nextgroupplot[
title={$T=300$},
scale only axis,
xmode=log,
xmin=0.0001,
xmax=10,
xminorticks=true,
ymin=-0.05,
ymax=1.1,
xlabel={$r$},
axis background/.style={fill=white},
axis x line*=bottom,
axis y line*=left,
]

\addplot [smooth,color=red, dotted,line width=2.0pt]
  table[row sep=crcr]{%
0.0001	0.81\\
0.000359381366380463	0.69\\
0.00129154966501488	0.4\\
0.00464158883361278	0.1\\
0.0166810053720006	0\\
0.0599484250318941	0\\
0.215443469003188	0\\
0.774263682681127	0\\
2.78255940220713	0\\
10	0\\
};
\addplot [smooth,color=blue, line width=2.0pt]
  table[row sep=crcr]{%
0.0001	0\\
0.000359381366380463	0\\
0.00129154966501488	0\\
0.00464158883361278	0\\
0.0166810053720006	0\\
0.0599484250318941	0\\
0.215443469003188	0.01\\
0.774263682681127	0\\
2.78255940220713	0.01\\
10	0.1\\
};

\addplot [smooth,color=cyan, dashed,line width=2.0pt]
  table[row sep=crcr]{%
0.0001	0\\
0.000359381366380463	0\\
0.00129154966501488	0\\
0.00464158883361278	0\\
0.0166810053720006	0\\
0.0599484250318941	0.02\\
0.215443469003188	0.4\\
0.774263682681127	0\\
2.78255940220713	0\\
10	0.05\\
};

\addplot [smooth,color=orange, line width=2.0pt]
  table[row sep=crcr]{%
0.0001	0.87\\
0.000359381366380463	0.87\\
0.00129154966501488	0.87\\
0.00464158883361278	0.87\\
0.0166810053720006	0.87\\
0.0599484250318941	0.87\\
0.215443469003188	0.87\\
0.774263682681127	0.87\\
2.78255940220713	0.87\\
10	0.87\\
};

\nextgroupplot[
title={$T=500$},
scale only axis,
xmode=log,
xmin=0.0001,
xmax=10,
xminorticks=true,
ymin=-0.05,
ymax=1.1,
xlabel={$r$},
axis background/.style={fill=white},
axis x line*=bottom,
axis y line*=left,
]              
   \addplot [smooth,color=red, dotted,line width=2.0pt]
  table[row sep=crcr]{%
0.0001	0.74\\
0.000359381366380463	0.6\\
0.00129154966501488	0.27\\
0.00464158883361278	0.01\\
0.0166810053720006	0\\
0.0599484250318941	0\\
0.215443469003188	0\\
0.774263682681127	0\\
2.78255940220713	0\\
10	0\\
};
\addplot [smooth,color=blue,line width=2.0pt]
  table[row sep=crcr]{%
0.0001	0\\
0.000359381366380463	0\\
0.00129154966501488	0\\
0.00464158883361278	0\\
0.0166810053720006	0\\
0.0599484250318941	0\\
0.215443469003188	0\\
0.774263682681127	0.03\\
2.78255940220713	0.02\\
10	0.06\\
};
\addplot [smooth,color=cyan, dashed,line width=2.0pt]
  table[row sep=crcr]{%
0.0001	0\\
0.000359381366380463	0\\
0.00129154966501488	0\\
0.00464158883361278	0\\
0.0166810053720006	0\\
0.0599484250318941	0.0\\
0.215443469003188	0.2\\
0.774263682681127	0.1\\
2.78255940220713	0\\
10	0.05\\
};
\addplot [smooth,color=orange, line width=2.0pt]
  table[row sep=crcr]{%
0.0001	0.85\\
0.000359381366380463	0.85\\
0.00129154966501488	0.85\\
0.00464158883361278	0.85\\
0.0166810053720006	0.85\\
0.0599484250318941	0.85\\
0.215443469003188	0.85\\
0.774263682681127	0.85\\
2.78255940220713	0.85\\
10	0.85\\
};

\end{groupplot}
\node at ($(group c2r1) + (0,-2.8cm)$) {\ref{grouplegend}}; 
\end{tikzpicture}%
}
    \caption{\small Empirical out-of-sample disappointment %$\mb P_\theta( \widehat c_r(\widehat x_r(\widehat \theta_T),\widehat \theta_T) < c(\widehat x_r(\widehat \theta_T),\theta_{\text{true}}))$ for 
    of different data-driven decisions based on synthetic data, computed using 100 independent training datasets.}
    \label{fig:synth:re}
\end{figure}

% \begin{figure}[h!]
%     \centering
%     \scalebox{0.55}{\input{outperf_instances}}   
%     \caption{\small Mean out-of-sample risk %$c(\widehat x_r(\widehat{\theta}_T), \theta_{\text{true}})$ 
%     of different data-driven decisions based on synthetic data. Solid lines represent empirical averages computed using 25 independent training datasets across $10$ random problem instances.}
%     \label{fig:synth:out-new}
% \end{figure}

\textbf{Real-world data.} The second experiment is based on a marketing dataset from Kaggle,\footnote{\url{https://www.kaggle.com/khalidnasereddin/retail-dataset-analysis}} which tracks the purchasing behavior of $2{,}000$ customers with respect to $d=5$ brands of chocolates. The customers are clustered into $n=5$ groups based on their age, education level and income by using the $K$-means++ clustering algorithm by~\citet{vassilvitskii2006k}. The resulting clusters readily induce a weight vector~$w$. The price vector~$a$ is sampled randomly from the uniform distribution on~$\{1,\ldots,10\}^d$. In addition, we concatenate the purchase histories of all customers in any group~$k$ and interpret the resulting time series as a trajectory of the unknown Markov chain corresponding to group~$k$.

\begin{figure}[h!]
    \centering
    \scalebox{0.55}{% This file was created by matlab2tikz.
\begin{tikzpicture}

\begin{groupplot}[group style = {group size = 3 by 1, horizontal sep = 30pt}, width=1.5in,height=1.15in]

\nextgroupplot[
title={$T=10$},
scale only axis,
xmode=log,
xmin=0.0001,
xmax=10,
xminorticks=true,
ymin=-7,
ymax=-5,
xlabel={$r$},
axis background/.style={fill=white},
axis x line*=bottom,
axis y line*=left,
]

\addplot [smooth,color=blue, mark=square, line width=2.0pt]
  table[row sep=crcr]{%
0.0001	-6.30613620685897\\
0.000359381366380463	-6.30613620685897\\
0.00129154966501488	-6.30613620685897\\
0.00464158883361278	-6.30613620685897\\
0.0166810053720006	-6.30613620685897\\
0.0599484250318941	-6.30613620685897\\
0.215443469003188	-6.30613620685897\\
0.774263682681127	-6.30613620685897\\
2.78255940220713	-6.30613620685897\\
10	-6.30613620685897\\
};

\addplot [color=red, mark=o,line width=2.0pt]
  table[row sep=crcr]{%
0.0001	-5.45356586804639\\
0.000359381366380463	-5.45356586804639\\
0.00129154966501488	-5.45356586804639\\
0.00464158883361278	-5.45356586804639\\
0.0166810053720006	-5.45356586804639\\
0.0599484250318941	-5.45356586804639\\
0.215443469003188	-5.45356586804639\\
0.774263682681127	-5.45356586804639\\
2.78255940220713	-5.45356586804639\\
10	-5.45356586804639\\
};

\addplot [smooth,color=orange, mark=triangle,line width=2.0pt]
  table[row sep=crcr]{%
0.0001	-5.45356586804639\\
0.000359381366380463	-5.45356586804639\\
0.00129154966501488	-5.45356586804639\\
0.00464158883361278	-5.45356586804639\\
0.0166810053720006	-5.45356586804639\\
0.0599484250318941	-5.45356586804639\\
0.215443469003188	-5.45356586804639\\
0.774263682681127	-5.45356586804639\\
2.78255940220713	-5.45356586804639\\
10	-5.45356586804639\\
};

\addplot [color=cyan, mark=star, line width=2.0pt]
  table[row sep=crcr]{%
0.0001	-6.30613620685897\\
0.000359381366380463	-6.30613620685897\\
0.00129154966501488	-6.30613620685897\\
0.00464158883361278	-6.30613620685897\\
0.0166810053720006	-6.30613620685897\\
0.0599484250318941	-6.30613620685897\\
0.215443469003188	-6.30613620685897\\
0.774263682681127	-6.30613620685897\\
2.78255940220713	-6.30613620685897\\
10	-5.59779127008488\\
};

\nextgroupplot[
title={$T=300$},
scale only axis,
xmode=log,
xmin=0.0001,
xmax=10,
xminorticks=true,
ymin=-7,
ymax=-5,
xlabel={$r$},
axis background/.style={fill=white},
axis x line*=bottom,
axis y line*=left,
legend style = { column sep = 7pt, legend columns = -1, legend to name = grouplegend,}
]
\addplot [color=blue, mark=square, line width=2.0pt]
  table[row sep=crcr]{%
0.0001	-6.30613620685897\\
0.000359381366380463	-6.30613620685897\\
0.00129154966501488	-6.30613620685897\\
0.00464158883361278	-6.30613620685897\\
0.0166810053720006	-6.30613620685897\\
0.0599484250318941	-6.30613620685897\\
0.215443469003188	-6.30613620685897\\
0.774263682681127	-6.30613620685897\\
2.78255940220713	-5.96869633890966\\
10	-6.30613620685897\\
};\addlegendentry{Markov DRO}
\addplot [smooth,color=cyan, mark=star, line width=2.0pt]
  table[row sep=crcr]{%
0.0001	-6.30613620685897\\
0.000359381366380463	-6.30613620685897\\
0.00129154966501488	-6.30613620685897\\
0.00464158883361278	-6.30613620685897\\
0.0166810053720006	-6.30613620685897\\
0.0599484250318941	-6.30613620685897\\
0.215443469003188	-6.30613620685897\\
0.774263682681127	-5.90988749320716\\
2.78255940220713	-6.30613620685897\\
10	-5.59779127008488\\
};\addlegendentry{Wasserstein}

\addplot [smooth,color=red, mark=o,line width=2.0pt]
  table[row sep=crcr]{%
0.0001	-5.45356586804639\\
0.000359381366380463	-5.45356586804639\\
0.00129154966501488	-5.45356586804639\\
0.00464158883361278	-5.45356586804639\\
0.0166810053720006	-5.45356586804639\\
0.0599484250318941	-5.45356586804639\\
0.215443469003188	-5.45356586804639\\
0.774263682681127	-5.45356586804639\\
2.78255940220713	-5.45356586804639\\
10	-5.45356586804639\\
};\addlegendentry{i.i.d.~DRO}

\addplot [smooth,color=orange, mark=triangle,line width=2.0pt]
  table[row sep=crcr]{%
0.0001	-5.45356586804639\\
0.000359381366380463	-5.45356586804639\\
0.00129154966501488	-5.45356586804639\\
0.00464158883361278	-5.45356586804639\\
0.0166810053720006	-5.45356586804639\\
0.0599484250318941	-5.45356586804639\\
0.215443469003188	-5.45356586804639\\
0.774263682681127	-5.45356586804639\\
2.78255940220713	-5.45356586804639\\
10	-5.45356586804639\\
};\addlegendentry{SAA}

\nextgroupplot[
title={$T=500$},
scale only axis,
xmode=log,
xmin=0.0001,
xmax=10,
xminorticks=true,
ymin=-7,
ymax=-5,
xlabel={$r$},
axis background/.style={fill=white},
axis x line*=bottom,
axis y line*=left,
]              
    
\addplot [smooth,color=blue, mark=square, line width=2.0pt]
table[row sep=crcr]{%
0.0001	-6.30613620685897\\
0.000359381366380463	-6.30613620685897\\
0.00129154966501488	-6.30613620685897\\
0.00464158883361278	-6.30613620685897\\
0.0166810053720006	-6.30613620685897\\
0.0599484250318941	-6.30613620685897\\
0.215443469003188	-6.30613620685897\\
0.774263682681127	-6.30613620685897\\
2.78255940220713	-6.30613620685897\\
10	-6.30613620685897\\
};
\addplot [smooth,color=cyan, mark=star, line width=2.0pt]
  table[row sep=crcr]{%
0.0001	-6.30613620685897\\
0.000359381366380463	-6.30613620685897\\
0.00129154966501488	-6.30613620685897\\
0.00464158883361278	-6.30613620685897\\
0.0166810053720006	-6.30613620685897\\
0.0599484250318941	-6.30613620685897\\
0.215443469003188	-6.30613620685897\\
0.774263682681127	-5.90988749320716\\
2.78255940220713	-6.30613620685897\\
10	-5.59779127008488\\
};
\addplot [color=red, mark=o,line width=2.0pt]
table[row sep=crcr]{%
0.0001	-5.45356586804639\\
0.000359381366380463	-5.45356586804639\\
0.00129154966501488	-5.45356586804639\\
0.00464158883361278	-5.45356586804639\\
0.0166810053720006	-5.45356586804639\\
0.0599484250318941	-5.45356586804639\\
0.215443469003188	-5.45356586804639\\
0.774263682681127	-5.45356586804639\\
2.78255940220713	-5.45356586804639\\
10	-5.45356586804639\\
};

\addplot [smooth,color=orange, mark=triangle,line width=2.0pt]
  table[row sep=crcr]{%
0.0001	-5.45356586804639\\
0.000359381366380463	-5.45356586804639\\
0.00129154966501488	-5.45356586804639\\
0.00464158883361278	-5.45356586804639\\
0.0166810053720006	-5.45356586804639\\
0.0599484250318941	-5.45356586804639\\
0.215443469003188	-5.45356586804639\\
0.774263682681127	-5.45356586804639\\
2.78255940220713	-5.45356586804639\\
10	-5.45356586804639\\
};
\end{groupplot}
\node at ($(group c2r1) + (0,-2.8cm)$) {\ref{grouplegend}}; 
\end{tikzpicture}%
}
    \caption{\small Out-of-sample risk %$c(\widehat x_r(\widehat{\theta}_T), \theta_{\text{test}})$ for 
    of different data-driven decisions based on 200 test samples from the Kaggle dataset.}
    \label{fig:real:out}
\end{figure}
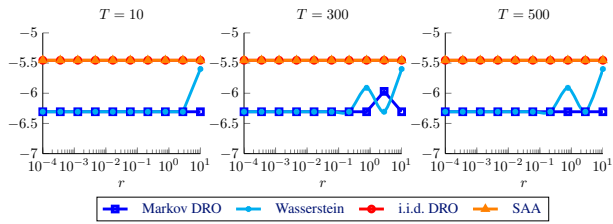
\textbf{Results.} Figures~\ref{fig:synth:out} and~\ref{fig:real:out} show that the proposed Markov DRO method outperforms the SAA scheme and the DRO approach by~\citet{ref:vanParys:fromdata-17} tailored to i.i.d.\ data (denoted as `i.i.d.~DRO') in that its out-of-sample risk displays a smaller mean as well as a smaller variance. For small training sample sizes~$T$, our method is also superior to the Wasserstein DRO approach by~\citet{ref:Mannor-20}, but for~$T\gtrsim 300$ the two methods display a similar behavior. %Figure~\ref{fig:real:out} visualizes a similar behaviour for the described real-world dataset.
Note also that the mean as well as the variance of the out-of-sample risk increase with~$r$ for all three DRO approaches. This can be explained by the increasing inaccuracy of a model that becomes more and more pessimistic.
%We remark that the discrepency between DRO approaches with Markovian assumptions and SAA is a result of the unobserved transitions in the aforementioned dataset. For DRO with Markovian assumptions, we put a small mass in front of the zero doublet frequency estimator so as to avoid division by zero when transforming to the transition kernel, whereas for SAA we simply count the frequency of single state observations. This initial discrepency in stationary distribution estimation may be magnified or offset depending on the problem instance. 
In fact, if no transitions are observed from a certain brand A to another brand B given that the data follows an ergodic Markov process, then brand B's true long-term value is overestimated when the data is falsely treated as i.i.d. This phenomenon explains the discrepancy between the out-of-sample risk incurred by methods that treat the data as Markovian or as i.i.d., respectively; see Figure~\ref{fig:real:out}. Note also that if the prescribed decay rate~$r$ drops to~$0$, then all methods should collapse to the SAA approach
if $\theta'>0$. If $\theta'$ has zero entries, however, the worst-case expected risk with respect to a conditional relative entropy ambiguity set does not necessarily reduce to the empirical risk if $r$ tends to $0$. We shed more light on this phenomenon in Remark~\ref{rmk:thetap:0}. %For example, not observing any transition from brand A to brand B would either result in an over-estimation of the transition probability $\mb P(\xi_t=B|\xi_{t-1}=A)$ or an under-estimation, and would end up increase/decrease the mass of the stationary distribution value for brand B, respectively. Suppose we are with the case where we decrease the stationary distribution value than the empirical estimator based on i.i.d. assumption. Then higher selling price for brand B would result in i.i.d.-based approaches showing a strong preference towards brand B while Markov-based approaches would not be at the same level of preference for brand B. In fact, if no transitions are observed from A to B given the data behaves like a Markov process, then brand B's long-term value is indeed overvalued by those i.i.d.-based approaches, and this situation explains why there is the discrepency when $r\to0$ as a consequence of the reformulation.
Figure~\ref{fig:synth:re} further shows that the Markov DRO approach results in the smallest out-of-sample disappointment among all tested methods for a wide range of rate parameters~$r$. Comparing the charts for~$T=10$, $T=300$ and~$T=500$, we also see that the out-of-sample disappointment of the Markov DRO method decays with~$T$, which is consistent with Theorem~\ref{thm:DRO:guarantees}. In practice, the optimal choice of the decay rate~$r$ remains an open question. % in the more general context of time-dependent data. 
To tune~$r$ with the aim to trade off in-sample performance against out-of-sample disappointment, one could use the rolling window heuristic for model selection in time series models by~\citet{bergmeir2012use}.

\textbf{Scalability.}
We also compare the scalability of the proposed Frank-Wolfe algorithm against that of an interior point method by \citet{waltz2006interior} (with or without exact gradient information), which represents a state-of-the-art method in nonlinear optimization. % in terms of elapsed time. 
To this end, we solve $10$ instances of the worst-case expectation problem~\eqref{eq:ex:Markov:sub} with rate parameter~$r=1$ for a fixed decision~$x$ sampled from the uniform distribution on~$\{0,1\}^d$. The 10 instances involve independent training sets of size~$T=5{,}000$, each of which is sampled from the same fixed Markov chain constructed as in the experiments with synthetic data. The problem instances are modelled in MATLAB, and all experiments are run on an Intel~i5-5257U CPU~(2.7GHz) computer with 16GB RAM. Figure~\ref{fig:wall:clock:time} shows that our Frank-Wolfe algorithm is significantly faster than both baseline methods whenever the Markov chain accommodates at least $d=100$ states. In particular, note that the interior point methods run out of memory as soon as~$d$ exceeds $200$. \\

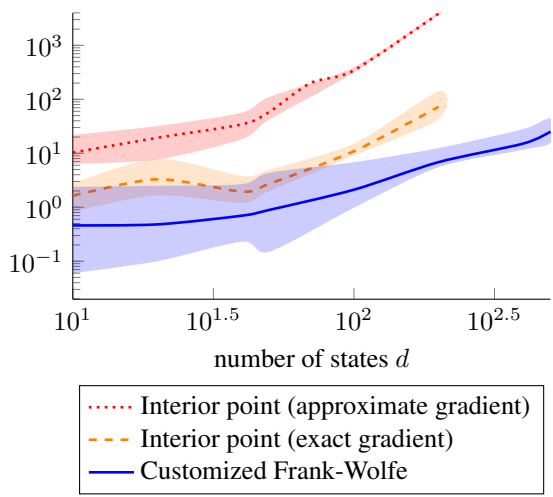
\begin{figure}[htb] 
   \begin{tikzpicture}

\begin{axis}[%
width=2.5in,
height=1.5in,
at={(-0,0)},
scale only axis,
xmode=log,
ymode=log,
xmin=10,
xmax=500,
xminorticks=true,
ymin=0,
ymax=4000,
xlabel={number of states $d$},
axis background/.style={fill=white},
axis x line*=bottom,
axis y line*=left,
legend cell align={left},
legend style={at={(0.5,-0.3)},anchor=north}
]
\addplot [smooth,color=red,dotted, line width=1]
  table[row sep=crcr]{%
10  10.294\\
20  19.4\\
40  34.512\\
50  58.407\\
70  200\\
100 345.374\\
200 3869.1\\
};
\addlegendentry{Interior point (approximate gradient)};

\addplot[smooth,draw=none, fill=red, fill opacity=0.2,forget plot]
table[row sep=crcr] {%
x	y\\
10	6.37\\
20  7.67\\
40  18.34\\
50  42.13\\
100 313.79\\
200 3646.4\\
200 4108.1\\
100 395.8\\
50  110.87\\
40  57.06\\
20  32.04\\
10	21.9\\
}--cycle;

\addplot [smooth,color=orange,dashed, line width=1]
  table[row sep=crcr]{%
10  1.628\\
20  3.301\\
40  1.953\\
50  2.72\\
100 10.931\\
200 72.745\\
};
\addlegendentry{Interior point (exact gradient)};

\addplot[smooth,draw=none, fill=orange, fill opacity=0.2,forget plot]
table[row sep=crcr] {%
x	y\\
10	0.8300\\
20  1.6600\\
40  1.2000\\
50  2\\
100 8.9700\\
200 42.9600\\
200 140.5800\\
100 15.1800\\
50  4.9800\\
40  3.8100\\
20  7.9000\\
10	2.7900\\
}--cycle;

\addplot [smooth,color=blue,solid, line width=1]
  table[row sep=crcr]{%
10  0.461\\
20  0.481\\
40  0.699\\
50  0.892\\
100 2.113\\
200 6.858\\
400 15.115\\
500 25.227\\
};
\addlegendentry{Customized Frank-Wolfe};
\addplot[smooth,draw=none, fill=blue, fill opacity=0.2,forget plot]
table[row sep=crcr] {%
x	y\\
10	0.06\\
20  0.1\\
40  0.2300\\
50  0.1500\\
100 0.9800\\
200 5.5300\\
400 11.5800\\
500 15.6500\\
500 44.1300\\
400 28.0900\\
200 12.4700\\
100 6.8400\\
50  4.3700\\
40  2.6900\\
20  2.5\\
10	2.39\\
}--cycle;
\end{axis}
\end{tikzpicture}% 
    \caption{Runtimes (in seconds) of different methods for solving problem~\eqref{def:DRO}. Lines represent means and shaded areas represent the ranges between the smallest and the largest runtimes.}
    \label{fig:wall:clock:time} 
\end{figure}

\textbf{Acknowledgements.} We thank Bart Van Parys for valuable comments on the paper and for suggesting the Markov coin example in Appendix~\ref{section:MCoin:example}. This research was supported by the Swiss National Science Foundation under the NCCR Automation, grant agreement~51NF40\_180545.

%\printbibliography[title={References}]
% \pagebreak
% \clearpage
\bibliography{ref.bib}
\bibliographystyle{icml2021}

%%%%%%%%%%%%%%%%%
%%%%%%%%%%%%%%%%%
%%%%%%%%%%%%%%%%%
\onecolumn
%\widetext
\newpage
\appendix
% \twocolumn[
\paragraph*{Additional notation.}
The $(i,j)$-th minor of a square matrix $A$ is defined as the determinant of a $A$ without its $i$-th row and $j$-th column. Similarly, the $(i\ell, jk)$-th minor of $A$ is defined as the determinant of $A$ without its $i$-th and $\ell$-th rows and without its $j$-th and $k$-th columns. The vector of all ones is denoted as~$\mathbf{1}$. Its dimension will always be clear from the context.
% Finally, If we write $f = o(g(x)), x \rightarrow 0$ then it means that $\lim _{x \rightarrow 0}\left|f(x)/g(x)\right|=0$. 

\section{Auxiliary Results for  Section~\ref{sec:problem:statement}} \label{app:sec:markov:chains}

\begin{myprop}[Properties of the conditional relative entropy] \label{prop:properties:conditional:entropy}
The conditional relative entropy $\Dc{\theta'}{\theta}$ introduced in Definition~\ref{def:conditional_relative_entropy} has the following properties.
\begin{enumerate}[label = (\roman*)]\setlength\itemsep{-0em}
\item If $\theta'\in\Theta'$ and $\theta \in \cl \Theta$, then $\Dc{\theta'}{\theta} \geqslant 0$. \label{nonneg}
\item If $\theta'\in\Theta$ and $\theta \in \cl \Theta$, then $\Dc{\theta'}{\theta}=0$ if and only if $\theta'=\theta$. \label{indiscernibles}
% \item{There exists a function $g:\mb R_+\to\mb R_+$ with $\lim_{x \to 0} g(x)=0$ such that $\| \theta'-\theta \| \leqslant g(\varepsilon)$ for all~$\varepsilon>0$, $\theta\in\Theta$ and $\theta'\in\Theta'$ with $\Dc{\theta'}{\theta} \leqslant \varepsilon$. \label{asympprop}}
\item $\Dc{\theta'}{\theta}$ is convex in $\theta'$ for every fixed $\theta\in\Theta$. \label{convex1prop}
\item $\Dc{\theta'}{\theta}$ is jointly continuous in $(\theta',\theta)$ on $\Theta'\times\Theta$. \label{contprop}
%\item $D_c$ is lower semi-continuous in $(\theta',\theta)$. \label{lscprop}
\end{enumerate}
\end{myprop}

\begin{proof} 
{The non-negativity of the conditional relative entropy $\Dc{\theta'}{\theta} =\sum_{i\in\Xi} (\pi_{\theta'})_i \, \D{(P_{\theta'})_{i\cdot}}{(P_\theta)_{i\cdot}}$ follows directly from the non-negativity of the relative entropy \citep[Theorem~2.6.3]{ref:Cover-06}, and thus Assertion~\ref{nonneg} follows. 

Next, fix any $\theta'\in\Theta$ and~$\theta\in\cl\Theta$, and note that $(\pi_{\theta'})_i>0$ for every~$i\in\Xi$. As the relative entropy vanishes if and only if its arguments coincide, we have~$\Dc{\theta'}{\theta}=0$ if and only if~$P_\theta=P_{\theta'}$. %\D{(P_{\theta'})_{i\cdot}}{(P_\theta)_{i\cdot}}=0$
Finally, as~$P_{\theta}\in\R^{d\times d}_{++}$ induces a unique stationary distribution~$\pi_{\theta}\in\R^d_{++}$ for every~$\theta'\in\Theta$, we find that~$\Dc{\theta'}{\theta}=0$ if and only if~$\theta={\theta'}$. Thus, Assertion~\ref{indiscernibles} follows.

As for \ref{convex1prop}, note that $t\log(t/x)=-t\log(x/t)$ is jointly convex in~$x>0$ and~$t>0$ as it represents the perspective of the convex function~$-\log(x)$. As convexity is preserved under affine transformations of the inputs, we may conclude that~$\theta'_{ij} \log( \theta'_{ij}/\sum_{k\in\Xi} \theta'_{ik})$ is convex in~$\theta'\in\Theta$. Recall that $\theta'_{ij}>0$ and $\sum_{k\in\Xi} \theta'_{ik}>0$ for all~$\theta'\in\Theta$. As these statements are true for all~$i,j,k\in\Xi$ and as convexity is preserved under sums, the claim follows.

To prove Assertion \ref{contprop}, note that we can rewrite the conditional relative entropy as
\begin{equation}
	\Dc{\theta'}{\theta} =\sum_{i,j=1}^d \theta'_{ij} \log\left( \frac{\theta'_{ij}}{\theta_{ij}} \right) -  \sum_{i,j=1}^d\theta'_{ij} \log\left( \sum_{k\in\Xi} \theta'_{ik}\right)+\sum_{i,j=1}^d\theta'_{ij} \log\left(\sum_{k=1}^d \theta_{ik}\right) ,
\end{equation}
where all terms are manifestly continuous in $(\theta',\theta)$ on $\Theta'\times\Theta$ thanks to our standard conventions for the logarithm.
%Finally, \ref{lscprop} follows from the joint lower semi-continuity of the relative entropy and the fact that lower semi-continuous is preserved by composition of lower semi-continuous functions with continuous functions.
}
\end{proof}

The following example highlights that, perhaps surprisingly, Assertion~\ref{indiscernibles} cannot be extended to arbitrary $\theta'\in\Theta'$.

\begin{myrmk}[Identity of indiscernibles]\label{rmk:thetap:0}
Assume that~$d=2$, and select any $\theta'\in\Theta'$ with $\theta'_{12}=\theta'_{21}=0$ and~$\theta'_{11}, \theta'_{22}>0$. These assumptions imply that~$\theta'\notin\Theta$. The conditional relative entropy distance between $\theta'$ and any $\theta\in\cl\Theta$ amounts to
\begin{align*}
    \Dc{\theta'}{\theta}
    =\theta'_{11}\log\left(\frac{\theta_{11}+\theta_{12}}{\theta_{11}}\right)+\theta'_{22}\log\left(\frac{\theta_{21}+\theta_{22}}{\theta_{22}}\right).
\end{align*}
Thus, we have $\Dc{\theta'}{\theta}=0$ whenever
$\theta_{12}=\theta_{21}=0$ and $\theta_{11}+\theta_{22}=1$ with~$\theta_{11},\theta_{22}>0$. The conditional relative entropy ambiguity set~$\{\theta\in\cl\Theta:\Dc{\theta'}{\theta}\leq r\}$ used in~\eqref{eq:DRO:general} thus contains infinitely many models even if~$r=0$, and it is misleading to picture this ambiguity set as a ball in~$\R^{d\times d}$. Moreover, if~$\theta'\notin\Theta$, then the worst-case expectation~\eqref{eq:DRO:general} does not reduce to the empirical risk~$\sum_{i,j=1}^d L(x,i)\theta'_{ij}$. This is in stark contrast to ambiguity sets commonly used for i.i.d.~data. Figure~\ref{fig:contour} visualizes a projection of the conditional relative entropy ambiguity set into the~$(\theta_{11},\theta_{22})$-plane for~$\theta'=\frac{1}{2} I$ and different radii~$r$ and shows that this set collapses to the line segment~$\{(\theta_{11},\theta_{22})\in\R^2_+:\theta_{11}+\theta_{22}=1\}$ as $r$ drops to~$0$.
%This observation is visualized by Figure~\ref{fig:contour}, that shows how the projection of the conditional relative entropy ambiguity set $\{\theta_{11},\theta_{22}\in[0,1] : \theta_{11}+\theta_{22}=1, \Dc{\theta'}{\theta}\leq r\}$ shrinks to a straight line $\theta_{11}+\theta_{22}=1$ as $r$ tends to 0.

% This result shows that if $\theta'$ has zero entries, forcing  $\Dc{\theta'}{\theta}=0$ does not necessarily induce a unique $\theta$, and henceforth can explain the discrepancy between DRO based on conditional relative entropy and SAA for small radii seen in the numerical examples.
\begin{figure}[h!]
\centering
        \includegraphics[width=0.22\linewidth]{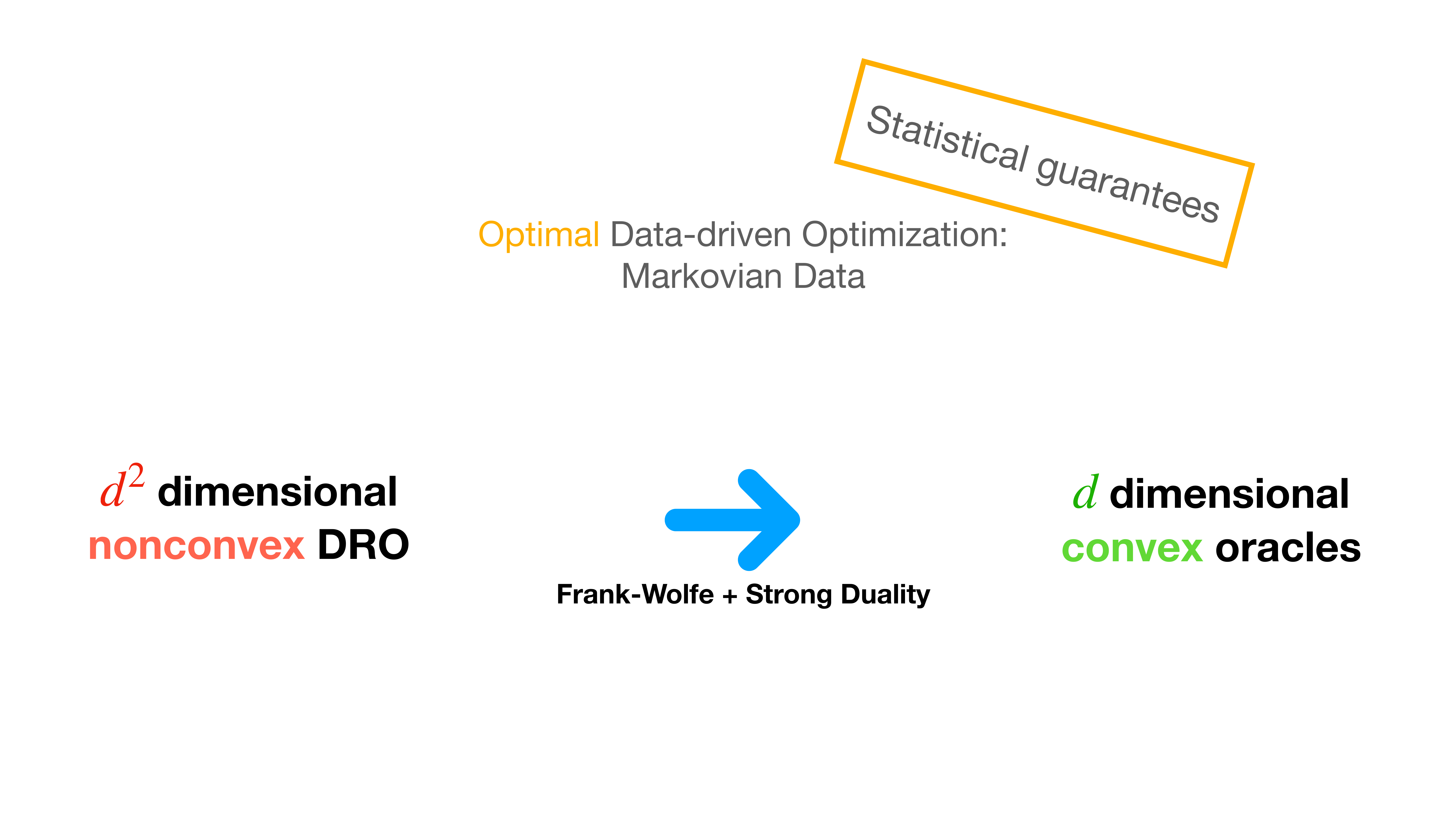}\vspace{-10pt}
        \caption{Projection of the conditional relative entropy ambiguity set into the $(\theta_{11}, \theta_{22})$-plane for~$\theta'=\frac{1}{2}I$ and $r=10^{k}$, where $k=-3$ (dark red), $k=-2$ (light red), $k=-1$ (orange), $k=0$ (yellow) and $k=1$ (green). This ambiguity set shrinks to the line segment between the two points $(0,1)$ and $(1,0)$ as the radius~$r$ tends to~$0$.}
        \label{fig:contour}
        \vspace{-10pt}
\end{figure}
\end{myrmk}

\begin{myrmk}[Nonconvexity of problem~\eqref{def:DRO}]\label{rmk:nonconvexity}
    %We now construct an instance of problem~\eqref{def:DRO} with a nonconvex feasible set. To this end, 
    Fix the estimator realization 
    \[
        \theta'=\begin{bmatrix}
       0.2 & 0.375\\
       0.375 & 0.05
    \end{bmatrix}\in\Theta',
    \]
    define the function $f:\Theta\to\Re$ through $f( \theta)= \sum_{i,j} \theta'_{ij}(-\log(\theta_{ij}/\sum_k \theta_{ik}))$ and set~$\bar r=r- \sum_{i,j} \theta'_{ij}\log(\theta'_{ij}/\sum_k \theta'_{ik})$. Using this notation, the feasible set of problem~\eqref{def:DRO} can be reformulated concisely as $\left\{\theta\in\Theta : f(\theta)\leq \bar r\right\}$. Next, set
    \begin{align*}
     \theta_1=\begin{bmatrix}
        0.9 & 0.045\\
        0.045 & 0.01
    \end{bmatrix}\in\Theta \quad \text{and} \quad
    \theta_2=\begin{bmatrix}
        0.01 & 0.045\\
        0.045 & 0.9
    \end{bmatrix}\in\Theta.
    \end{align*}
An explicit calculation reveals that $f(\theta_1)\le f(\theta_2)\le 1.6$, whereas $f((\theta_1+\theta_2)/2)>1.6$. This implies that both~$\theta_1$ and~$\theta_2$ are feasible in problem~\eqref{def:DRO} if $r$ is chosen such that~$\bar r=1.6$, while their midpoint~$(\theta_1+\theta_2)/2$ is not. Hence, the feasible set of~\eqref{def:DRO} is generically nonconvex. This allows us to conclude that~\eqref{def:DRO} is generically a nonconvex problem.
\end{myrmk}
 \begin{myprop}[Reparametrized ambiguity set]\label{prop:D:convex}
	For any fixed $P_{\theta^\prime}\in \mc P$, $\pi_{\theta'}\in\Delta_d$ and $r>0$ the ambiguity set 
	\begin{equation}\label{eq:feasible:region:psi}
	{\mc D} = \left\{P\in{\mc P}: \sum_{i\in\Xi} (\pi_{\theta'})_i\D{(P_{\theta^\prime})_{i\cdot}}{P_{i\cdot}}\leq r\right\}
	\end{equation}
	is convex and compact.
\end{myprop}
\vspace{-10pt}
\begin{proof}
% [Proof for Proposition~\ref{prop:D:convex}]
	This is an immediate consequence of the convexity and compactness of the set~$\mathcal P$ of all row-stochastic matrices in~$\R^{d\times d}$ and the convexity and continuity of the relative entropy on~$\Delta_d\times\Delta_d$. 
	%By the convexity of $\mc P$, together with the convexity and continuity of the relative entropy, $\{p\in {\mc P}: \D{(P_{\theta^\prime})_{i\cdot}}{P_{i\cdot}}\leq \ell\}$ is convex and compact for all $\ell\geqslant 0$. Since for any $\pi_{\theta'}\in\Delta_d$, $\{p\in{\mc P}:\sum_{i\in\Xi} (\pi_{\theta'})_i \D{(P_{\theta^\prime})_{i\cdot}}{P_{i\cdot}}\le r\}$ is a convex hull of $\{p\in{\mc P}: \D{(P_{\theta^\prime})_{i\cdot}}{P_{i\cdot}} \leq \ell_0\}$ for some $\ell_0$, thus still convex and compact.
\end{proof}

%***************************************
%***************************************

%\begin{mylem}
%	If two transition kernels $P,P'$ corresponding to model $\theta,\theta'$ respectively has the relation $P=P'$, then we must have that the underlying model $\theta=\theta'.$
%\end{mylem}

\section{Asymptotic Consistency}

The estimator~$\widehat \theta_T$ may fail to have balanced marginals and thus fail to belong to~$\cl\Theta$. In this case, the worst-case expectation problem~\eqref{def:DRO} with~$\theta'=\widehat\theta_T$ may be infeasible for small~$r$. The following lemma, which is an important ingredient for the proof of Theorem~\ref{thm:asymptotic:consistency}, shows that problem~\eqref{def:DRO} is guaranteed to be feasible and solvable whenever~$r\geqslant d/T$. % about the asymptotic consistency, we need to ensure that problem \eqref{eq:DRO:general} is feasible. This requires the sequence of radii not to shrink too slowly and is quantified in the following lemma.
% % %************************************
\begin{mylem}[Solvability of \eqref{def:DRO}] \label{lem:feasiblity}
If~$\theta'$ is a realization of~$\widehat\theta_T$ for some~$T\in\mathbb N$ and~$r\geq \frac{d}{T}$, then problem~\eqref{def:DRO} is solvable.
\end{mylem}

\begin{proof}
Consider the modified estimator $\widetilde{\theta}_T$ defined through
\begin{equation*}
 (\widetilde{\theta}_{T})_{i j}=\frac{1}{T+1} \left(\sum_{t=1}^{T} \indic{\left(\xi_{t-1}, \xi_{t}\right)=(i, j)} +\indic{\left(\xi_{T}, \xi_{1}\right)=(i, j)} \right) \quad \forall i,j\in\Xi.   
\end{equation*}
Note that $\widetilde{\theta}_T$ differs from our standard estimator $\widehat \theta_T$ in that it accounts for an artificial `ghost' transition from $\xi_T$ to $\xi_1$, which ensures that $\widetilde{\theta}_T$ has balanced marginals, that is, $\widetilde{\theta}_T\in\cl \Theta$. In the following we show that the conditional relative entropy distance between the two estimators $\widehat \theta_T$ and $\widetilde{\theta}_T$ is small for large $T$. To see this, note first that
\begin{align*}
    \Dc{\widehat\theta_T}{\widetilde\theta_T}&=\sum_{i, j=1}^d (\widehat{\theta}_T)_{i j}\left(\log \frac{(\widehat{\theta}_T)_{i j}}{\sum_{k=1}^d (\widehat{\theta}_T)_{i k}}-\log \frac{(\widetilde{\theta}_T)_{i j}}{\sum_{k=1}^d (\widetilde{\theta}_T)_{i k}}\right)
    =\sum_{i, j=1}^d (\widehat{\theta}_T)_{i j}\left(\log \frac{(\widehat{\theta}_T)_{i j}}{(\widetilde{\theta}_T)_{i j}}\cdot\frac{\sum_{k=1}^d (\widetilde{\theta}_T)_{i k}}{\sum_{k=1}^d (\widehat{\theta}_T)_{i k}}\right).
\end{align*}
By using the definitions of~$\widehat \theta_T$ and~$\widetilde\theta_T$, for any fixed $i,j\in\Xi$ the logarithm in the last expression can be reformulated as
\begin{align*}
 \displaystyle  &\log \left( \frac{\frac{1}{T} (\sum_{t=1}^{T} \indic{\left(\xi_{t-1}, \xi_{t}\right)=(i, j)} )}{\frac{1}{T+1} (\sum_{t=1}^{T} \indic{\left(\xi_{t-1}, \xi_{t}\right)=(i, j)} +\indic{\left(\xi_{T}, \xi_{1}\right)=(i, j)})} \cdot \frac{\sum_{k=1}^d\frac{1}{T+1}(\sum_{t=1}^{T} \indic{\left(\xi_{t-1}, \xi_{t}\right)=(i, k)} +\indic{\left(\xi_{T}, \xi_{1}\right)=(i, k)})}{\sum_{k=1}^d\frac{1}{T} (\sum_{t=1}^{T} \indic{\left(\xi_{t-1}, \xi_{t}\right)=(i, k)} ) } \right) \\
 &\qquad=\log \left( \frac{\sum_{t=1}^{T} \indic{\left(\xi_{t-1}, \xi_{t}\right)=(i, j)} }{\sum_{t=1}^{T} \indic{\left(\xi_{t-1}, \xi_{t}\right)=(i, j)} +\indic{\left(\xi_{T}, \xi_{1}\right)=(i, j)}} \right)+ \log\left( \frac{\sum_{k=1}^d \sum_{t=1}^{T} \indic{\left(\xi_{t-1}, \xi_{t}\right)=(i, k)} +\sum_{k=1}^d \indic{\left(\xi_{T}, \xi_{1}\right)=(i, k)}}{\sum_{k=1}^d \sum_{t=1}^{T} \indic{\left(\xi_{t-1}, \xi_{t}\right)=(i, k)}  } \right)\\
  &\qquad\leq \log\left(1+\frac{\sum_{k=1}^d\indic{\left(\xi_{T}, \xi_{1}\right)=(i, k)}}{\sum_{k=1}^d \sum_{t=1}^{T} \indic{\left(\xi_{t-1}, \xi_{t}\right)=(i, k)}  }\right) = \log\left( 1 + \frac{\indic{\xi_T=i}}{T\sum_{k=1}^d(\widehat \theta_T)_{ik}} \right) \leq \frac{1}{T\sum_{k=1}^d(\widehat \theta_T)_{ik}}\ ,
\end{align*}
where the first inequality is obtained by omitting the first (non-positive) of the two logarithm terms, while the second inequality uses the elementary bounds~$\indic{\xi_T=i}\leqslant 1$ and $\log(1+t) \leq t$ for all $t\geqslant 0$. Substituting the resulting estimate into the above formula for the conditional relative entropy then yields
\begin{align*}
     \Dc{\widehat\theta_T}{\tilde\theta_T} \leq \sum_{i, j=1}^d (\widehat{\theta}_T)_{i j} \frac{1}{T\sum_{k=1}^d(\widehat \theta_T)_{ik}}
     = \sum_{i=1}^d \frac{1}{T} = \frac{d}{T}.
\end{align*}
As~$\widetilde \theta_T\in\cl\Theta$, this inequality implies that~$\widetilde{\theta}_T$ is feasible in the worst-case expectation problem~\eqref{def:DRO} if~$\theta'=\widehat\theta_T$ and~$r \geq \frac{d}{T}$. Hence, problem~\eqref{def:DRO} maximizes a continuous function over a non-empty compact set and is therefore solvable.
\end{proof}

\begin{proof}[Proof of Theorem~\ref{thm:asymptotic:consistency}]
%Fix any~$\theta\in\Theta$, and note that $\widehat{\theta}_T$ converges $\mb P_\theta$-almost surely to $\theta$ thanks to the ergodic theorem for Markov chains \citep[Theorem~4.1]{ross2010introduction}. In addition, by \citep[Proposition~3.1]{ref:Sutter-19}, the data-driven predictor $\widehat c_r(x,\theta')$ is continuous on~$X\times\Theta$. 
Fix any~$\theta\in\Theta$, and denote by $\theta^\star_{T}\in\cl \Theta$ a maximizer of problem~\eqref{def:DRO} with $\theta'=\widehat\theta_T$ and~$r=r_T$, which exists thanks to Lemma~\ref{lem:feasiblity} and because $r_T\geq d/T$. Thus, we have $\Dc{\widehat{\theta}_T}{\theta^\star_{T}} \leqslant r_T$ by feasibility and $\widehat c_{r_T}(x,\widehat\theta_T)= c(x,\theta^\star_{T})$ by optimality. Note that $\widehat{\theta}_T$ converges $\mb P_\theta$-almost surely to $\theta$ thanks to the ergodic theorem for Markov chains \citep[Theorem~4.1]{ross2010introduction}. Without loss of generality, we now focus on a realization of the stochastic process of training samples for which $\widehat{\theta}_T$ represents a deterministic sequence converging to $\theta$ with certainty. For this fixed realization, we can prove that~$\theta^\star_{T}$ converges to $\theta$, as well. To this end, assume for the sake of argument that~$\theta^\star_{T}$ does {\em not} converge to~$\theta$. As~$\cl \Theta$ is compact, this implies that the sequence~$\theta^\star_{T}$ has a cluster point~$\theta^\star\in\cl \Theta$ such that~$\theta^\star\neq\theta$. This in turn implies that there exists a subsequence $\{\theta^\star_{T_k}\}_{k\in\mathbb N}$ in~$\cl \Theta$ that converges to~$\theta^\star$. Thus, we find the contradiction
\begin{align*}
     0< \Dc{\theta}{\theta^\star} = \Dc{\lim_{k\to\infty}\widehat{\theta}_{T_k}}{\lim_{k\to\infty}\theta^\star_{T_k}} \leqslant \liminf_{k\to\infty} \Dc{\widehat{\theta}_{T_k}}{\theta^\star_{T_k}} \leqslant \liminf_{k\to\infty}r_{T_k}=0,
\end{align*}
where the first (strict) inequality follows from Proposition~\ref{prop:properties:conditional:entropy}\ref{nonneg} and the assumption that~$\theta\in\Theta$, the second inequality holds because the conditional relative entropy is lower semi-continuous on~$\Theta'\times\cl \Theta$, and the third inequality exploits the feasibility of~$\theta^\star_{T_k}$ in~\eqref{def:DRO} with~$r=r_{T_k}$. We may therefore conclude that~$\theta^\star_T$ converges to~$\theta$ as~$T$ tends to infinity.

Next, for any fixed decision $x\in X$ we have
$$ 
    \lim_{T\to\infty} \widehat c_{r_T}(x,\widehat \theta_{T})=\lim_{T\to\infty} c(x,\theta^\star_{T})= c(x,\lim_{T\to\infty}\theta^\star_{T})=c(x,\theta), %\quad \mb P_\theta  \text{-a.s.} ,
$$
where the first equality exploits the feasibility of~$\theta^\star_{T}$ in~\eqref{def:DRO} with~$r=r_{T}$, and the second equality holds because $c(x,\theta)$ is continuous (linear) in~$\theta$. As $\widehat{\theta}_T$ converges to $\theta$ for $\mb P_\theta$-almost every trajectory of the training samples, the claim follows.
\end{proof}

%%%%%%%%%%%%%%%%%%%%%%%%%%%%%%%%%%%%%
\section{Proofs and Background Results for Section~\ref{ssec:approximation:algorithm}} \label{app:sec:theoretical:guarantees}
This section discusses the theoretical foundations of Algorithms~\ref{alg:FW} and~\ref{alg:subprob}. We first provide a proof for Lemma~\ref{reparam}, which reformulates problem~\eqref{def:DRO} with a nonconvex feasible set as problem~\eqref{eq:lipschitz:DRO:problem} with a nonconvex objective function.

\begin{proof}[Proof of Lemma~\ref{reparam}]
We first show that the ambiguity set $\mc D_r(\theta')$ is a subset of~$\R^{d\times d}_{++}$. To this end, assume for the sake of argument that there is a transition probability matrix~$P\in\mc D_r(\theta')$ and two states~$i,j\in\Xi$ such that~$P_{ij}=0$. Then, we have
$$
    \infty= (P_{\theta'})^2_{ij}\log\left(\frac{(P_{\theta'})_{ij}}{P_{ij}} \right) \leqslant \sum_{k\in\Xi} (\pi_{\theta'})_k\D{(P_{\theta^\prime})_{k\cdot}}{(P)_{k\cdot}} \leqslant r<\infty,
$$
where the equality follows from our standard conventions for the logarithm and the assumptions that~$\theta'>0$, which implies that~$(P_{\theta'})_{ij}>0$. We have thus derived a contradiction, which allows us to conclude that~$P_{ij}>0$ for all~$i,j\in\Xi$. Choose now any~$P\in\mc D_r(\theta')$, and use~$\pi$ to denote the stationary distribution corresponding to~$P$. The above arguments imply that~$P>0$. In addition, the invariant distribution $\pi$ satisfies the following system of linear equations.
\begin{equation} \label{eq:defining:property:of:Ad}
	\begin{bmatrix}
		P_{11}-1 & P_{21} & \ldots  & P_{(d-1)1} & P_{d1}\\
		P_{12}& P_{22}-1 & \ldots & P_{(d-1)2}&  P_{d2}\\
		\vdots & \vdots & \ddots &\vdots & \vdots \\
		P_{1(d-1)}& P_{2(d-1)} & \ldots &P_{(d-1)(d-1)}-1&  P_{d(d-1)}\\
		1		&	1	&	 \ldots	& 1 & 1
			\end{bmatrix}
			\begin{bmatrix}
				\pi_1\\\pi_2\\\vdots \\\pi_{d-1}\\\pi_d
			\end{bmatrix}=\begin{bmatrix}
				0\\0\\\vdots \\0\\1
			\end{bmatrix}
\end{equation}
Indeed, the first $d-1$ equations are equivalent to the stationarity condition $\pi P=\pi$, which can be recast as $(P^\top-I)\pi^\top=0$, and the last equation represents the normalization condition $\sum_{i=1}^d \pi_i=1$. Next, denote the constraint matrix in~\eqref{eq:defining:property:of:Ad} by~${A}_d(P)$. The matrix $(P^\top-I)$ has rank $d-1$, and the matrix ${A}_d(P)$, which is obtained by replacing the last row of $(P^\top-I)$ with a row of ones, has full rank \citep[Theorem 4.16]{berman1994nonnegative}. Thus, the system of equations~\eqref{eq:defining:property:of:Ad} has a unique solution~$\pi$. Moreover, as~$P>0$, the Perron-Frobenius theorem~\citep{perron1907} guarantees that~$\pi>0$. Using the notation introduced so far, the objective function of problem~\eqref{def:DRO} can be reformulated as 
\begin{align*}
	c(x,\theta)=\mb E_{\theta}[L(x,\xi)]=%\sum_{i\in\Xi} L(x,i)\cdot \sum_{j\in\Xi} \theta_{ij}=
	\sum_{i=1}^d L(x,i)\pi_i=\sum_{i=1}^d L(x,i) ({A}_d(P)^{-1})_{id}=\psip,
\end{align*}
where the last equality follows from the definition of
$\psip$ in the main text. Thus, the claim follows.
\end{proof}

\begin{comment}
\begin{myrmk}[Gradient of $\Psi$]\label{rmk:gradient:psi}
The gradient $\nabla_P \Psi(x,P)$ of the reparametrized objective function $\Psi(x,P)$ with respect to~$P\in\R^{d\times d}$ can be derived from the formula for the gradient of the Perron-Frobenius vector with respect to the transition kernel~\citep[Theorem 4.1]{caswell2013sensitivity}. A direct calculation yields
\begin{equation}
    \label{eq:gradient}
    \nabla_P \Psi(x,P)=\sum_{i=1}^d L(x,i) (Z_{\cdot i}-\pi^\top )\pi,
    % \nabla_P \Psi(x,P)=\sum_{i=1}^d L(x,i)(\pi\otimes (Z-\pi^\top \mathbf{1}^\top)_{\cdot,i}),
\end{equation}
where $\pi_i=((A_{d}(P))^{-1})_{i d}$ for every~$i\in\Xi$, and where $Z=\left(I-P^\top+ \mathbf{1}\pi \right)^{-1}$ %$Z=\left(I-P+\pi^\top \mathbf{1}^\top \right)^{-\top}$
denotes the fundamental matrix of the ergodic Markov chain associated with $P$. 
% The operator `$\otimes$' stands for the tensor product.
\end{myrmk}
\end{comment}

\begin{myrmk}[Nonconvexity of problem \eqref{eq:lipschitz:DRO:problem}]\label{rmk:nonconvex}
The function $\Psi(x,P)$ is neither concave nor convex in~$P$ if $d\geqslant 2$.
To see this, consider an example with $d=2$, and assume that $L(x,1)\ne L(x,2)$. As $P\in\R^{2\times 2}$ is a row-stochastic matrix, we have $P_{12} = 1-P_{11}$ and $P_{22} = 1 - P_{21}$. These relations allow us to define an auxiliary function $\tilde \Psi$, through
\[
    \tilde \Psi(x,P_{11}, P_{21}) = \Psi\left(x,\begin{bmatrix}
        P_{11} & 1-P_{11}\\
        P_{21} & 1- P_{21}
    \end{bmatrix}\right).
\]
A direct calculation then shows that
\begin{align*}
	\det\left(\nabla_{(P_{11},P_{21})}^2\tilde\Psi(x,P_{11}, P_{21})\right)=-\frac{(L(x,1)-L(x,2))^2}{\left(1- P_{11}+ P_{21}\right)^4} < 0,
\end{align*} 
which indicates that the Hessian matrix of $\tilde\Psi(x,P_{11}, P_{21})$ has a positive and a negative eigenvalue. Hence, $\tilde\Psi(x,P_{11}, P_{21})$ is neither concave nor convex in $(P_{11}, P_{21})$, which implies that $\psip$ is neither concave nor convex in~$P$. 
\end{myrmk}

Despite being nonconvex, problem \eqref{eq:lipschitz:DRO:problem} displays desirable structural properties, which ensure that Algorithm~\ref{alg:FW} converges.

\begin{mylem}[Lipschitz continuous gradient]\label{lem:objlipschitzgrad}
	If $\theta'>0$, then $\nabla_P\psip$ is Lipschitz continuous in~$P$ on the ambiguity set~$\mc D_r(\theta') = \{ P\in \mc P: \sum_{i} (\pi_{\theta'})_i\D{(P_{\theta^\prime})_{i\cdot}}{(P)_{i\cdot}}\leqslant r\}$ for any fixed $x\in X$.
\end{mylem}

\begin{proof}
From the proof of Lemma~\ref{reparam} we know that~$P>0$ and $A_d(P)$ is invertible for any~$P\in\mc D_r(\theta)$. In addition, as~$\det(A_d(P))$ is continuous on the compact set~$\mc D_r(\theta')$, there exists~$\delta>0$ such that~$|\det(A_d(P))| > \delta$ for all~$P\in \mc D_r(\theta')$. To show that $\nabla_P\psip$ is Lipschitz continuous, it suffices to show that the gradient of~$\pi_k=(A_d(P)^{-1})_{kd}$ with respect to~$P$ is Lipschitz continuous for every~$k\in\Xi$. In the remainder of the proof, we use~$A$ as a shorthand for~$A_d(P)$. In addition, we use $A_{ij}$ to denote the $(i,j)$-th minor and $A_{ij,k\ell}$ to denote the $(ij, k\ell)$-th minor of $A$, respectively. Cramer's rule for computing inverse matrices then implies that $\pi_k=(-1)^{d+k}A_{dk}/\det(A)$.
%We have from Lemma~\ref{reparam} that $\pi_k=\left(A^{-1}\right)_{kd}. $ Writing out the formula for the inverse of a matrix we get 
% \begin{align*} 
% \pi_k=\frac{(-1)^{d+k}A_{dk}}{\det(A)}.  
% \end{align*}
The partial derivatives of~$\pi_k$ are given by
%Differentiating $\pi_k$ with respect to each coordinate $P_{ij}$ ($j<d$)  yields the following expression for $i=k$:
\begin{align*}
\frac{\partial \pi_k}{\partial P_{ij}}=\left\{\begin{array}{ll}
\displaystyle \frac{(-1)^{d+k+1}A_{dk}}{\det(A)^2}\left(\sum_{\ell\ne k} (-1)^{d+\ell} A_{dj,\ell k}\right) & \text{if $i=k$ and $j<d$},\\
\displaystyle \frac{(-1)^{d+k+1}A_{dk}}{\det(A)^2}\left(\sum_{\ell\ne i} (-1)^{d+\ell} A_{dj,\ell i}\right)+\frac{(-1)^{d+k}A_{dj,ki}}{\det(A)} & \text{if $i\neq k$ and $j<d$,}\\
0 & \text{if $j=d$,}
% \displaystyle -\sum_{\ell=1}^{d-1}\frac{\partial \pi_k}{ \partial P_{i\ell}} & 
\end{array}\right.
\end{align*}
where $\sum_{\ell\ne i} (-1)^{d+\ell} A_{dj,\ell i}$ represents the partial derivative of $\det(A)$ with respect to $P_{ij}$.
%For all $i\ne k$, we have 
%(-1)^{d+\ell} A_{dj,\ell i}\right)+\frac{(-1)^{d+k}A_{dj,ki}}{\\begin{align}\label{eq:step:2}
% 	  \frac{\partial \pi_k}{\partial P_{ij}}=\frac{(-1)^{d+k+1}A_{dk}}{\det(A)^2}\left(\sum_{\ell\ne i} det(A)}. %=\frac{Z(P)}{\det(A)^2},
% \end{align}
% where $Z(p)$ denotes the corresponding polynomial of $p$.
In all cases, the partial derivatives represent rational functions of~$P$, that is, fractions of polynomials. As all polynomials are Lipschitz continuous on compact sets and as the denominator polynomials of all rational functions are bounded away from~$0$, all partial derivatives are indeed Lipschitz continuous thanks to~\citep[Theorem 12.5]{Erikssonapplied2004}. Hence, $\nabla_P\psip$ is Lipschitz continuous.
\end{proof}

We now prove that Algorithm~\ref{alg:FW} converges to an approximate stationary point with an arbitrarily small Frank-Wolfe gap.

% \begin{mylem}[Convergence of Algorithm~\ref{alg:FW}]\label{thm:FW:guarantee}
% For any fixed $\varepsilon>0$ and $\theta'\in\Theta$, the Frank-Wolfe Algorithm~\ref{alg:FW} for solving problem~\eqref{eq:lipschitz:DRO:problem} finds an approximate stationary point with a Frank-Wolfe gap of at most $\varepsilon$ in $O(1/\varepsilon^2)$ iterations. 
% \end{mylem}

% \begin{proof}
% 	Consider $\mc P_{\theta'} = \{P_{\theta'}\in \mc P: (P_{\theta'})_{ij}>0 \ \forall i,j\}$. Then $\mc P_{\theta'}$ and the set of all stochastic matrices of dimension $d\times d$ clearly have equal Lebesgue measure. Moreover, $\mc P_{\theta'}$ is a subset of the set of all irreducible $P_{\theta'}$. Hence, the function $\Psi: X\times{\mc P} \to \mb R$ is indeed Lipschitz continuous in its gradient within the ambiguity set by Lemma~\ref{lem:objlipschitzgrad} for almost all $P_{\theta'}$. Hence, $\Psi$ is ensured to have a finite curvature constant $C_f$ as stated in \cite{lacoste2016convergence}. Moreover, we have shown in Proposition~\ref{prop:D:convex}, that the feasible region $\{P\in{\mc P}:\sum_{i\in\Xi} (\pi_{\theta'})_i \D{(P_{\theta'})_{i\cdot}}{(P_{\theta})_{i\cdot}}\leqslant r\}$ is convex and compact. Hence, applying \citep[Theorem 1]{lacoste2016convergence} gives the result.
% \end{proof}

\begin{proof}[Proof of Theorem~\ref{thm:global:convergence:FW}]
As~$\theta'>0$, Lemma~\ref{lem:objlipschitzgrad} implies that $\nabla_P \psip$ is Lipschitz continuous in~$P$ on the ambiguity set~$\mc D_r(\theta')$. Therefore, the objective function~$\psip$ has a finite curvature constant; see \citep[Lemma~7]{pmlr-v28-jaggi13}. As~$\mc D_r(\theta')$ is convex and compact thanks to Proposition~\ref{prop:D:convex}, the claim follows directly from \citep[Theorem 1]{lacoste2016convergence}.
%By Lemma~\ref{thm:FW:guarantee}, the Frank-Wolfe Algorithm~\ref{alg:FW} achieves a feasible solution with Frank-Wolfe gap smaller than $\varepsilon$ after $\mc O(1/\varepsilon)$ iterations. This in turn implies an approximate stationary point of $\Psi$, which concludes the proof.
\end{proof}

%Equipped with the convergence guarantees stated in Theorem~\ref{thm:global:convergence:FW}, we aim to solve the reparametrized problem~\eqref{eq:lipschitz:DRO:problem} via the Frank-Wolfe algorithm. 
The efficiency of Algorithm~\ref{alg:FW} depends on the efficiency of Algorithm~\ref{alg:subprob} for solving oracle subproblems of the form
\begin{align} \label{primal:problem:linearized}
	\mc J^\star=  \left\{
	\begin{array}{cl}\max \limits_{P\in{\mc P}} & \tr {C^\top P} \\ \st &\sum_{i=1}^d \alpha_i \D{P'_{i\cdot}}{P_{i\cdot}}\leq r,
	\end{array}\right.
\end{align}
where $P'=P_{\theta'}$ is the transition probability matrix corresponding to~$\theta'$, $C=\nabla_P \Psi(x,P^{(m)})$ is the gradient of the objective function at a fixed anchor point~$P^{(m)}$ and $\alpha_i=(\pi_{\theta'})_i$ is the invariant probability of state $i\in\Xi$ under model~$\theta'$. %Note that Algorithm~\ref{alg:subprob} is called in each iteration of Algorithm~\ref{alg:FW}. 
Note that~\eqref{primal:problem:linearized} represents a convex optimization problem thanks to Proposition~\ref{prop:D:convex}. In order to prove Theorem~\ref{thm:correct:algsub} from the main text, we first show that a maximizer for~\eqref{primal:problem:linearized} can be constructed highly efficiently by solving the problem dual to~\eqref{primal:problem:linearized}.

\begin{myprop}
[Linearized oracle subproblem]\label{prop:linear:predictor}
If~$\varepsilon>0$, then the following statements hold.
\begin{enumerate}[label = (\roman*)]\setlength\itemsep{-0em}
	\item \label{duality:thm:i}The strong dual of the linearized oracle subproblem \eqref{primal:problem:linearized} is given by 
	\begin{align}\label{dual:problem:linearized}
		\mc J^\star=\min_{\substack{\eta \in \mathbb{R}^{d},\lambda \geqslant 0\\\eta_i> \max_j\{C_{ij}\}}}  \lambda (r-1) + \sum_{i=1}^d \eta_i+\lambda \sum_{j=1}^{d}\alpha_i P'_{ij}
		\log\left(\frac{\lambda\alpha_i}{\eta_i-C_{ij}}\right) .
	\end{align}
	%Moreover, strong duality holds: $\mc J^\star_{primal} = \mc J^\star_{dual}$.
	\item \label{duality:thm:ii} If~$\eta\in\R^d$ is feasible in the dual problem \eqref{dual:problem:linearized}, then the corresponding optimal choice for $\lambda$ is given by 
\begin{align*}
		\lambda^\star(\eta)
%		=\exp \left(\sum_{i=1}^{d}\sum_{j=1}^{d} \alpha_i P'_{ij}\log\left(\frac{\eta_i-c^{i}_j}{\alpha_i}\right)-r\right)
%		=\sum_{i=1}^{d}\exp \left(\frac{1}{\alpha_i} \left( \sum_{j=1}^{d} \alpha_{i} P'_{ij} \log \left(\frac{\eta_i-C_{ij}}{\alpha_{i}}\right)-r\right)\right).
=\exp \left(\sum_{i,j=1}^{d} \alpha_{i} P'_{ij} \log \left(\frac{\eta_i-C_{ij}}{\alpha_{i}}\right)-r\right).
	\end{align*} 
	\item \label{duality:thm:iv} If $\eta^\star$ is optimal in the dual problem~\eqref{dual:problem:linearized}, then we have
    $$
        \max_j\{C_{ij}\} < \eta_i^\star \leqslant (1- e^{-r})^{-1} \left(d\max_{ij}\{C_{ij}\}-e^{-r}\tr {C^\top P'}\right) - \sum_{\substack{k=1\\k\ne i}}^d \max_j\{C_{kj}\} \quad\forall i\in\Xi.
    $$
\item \label{duality:thm:iii} If~$P^\star$ is optimal in the primal problem \eqref{primal:problem:linearized} and $(\lambda^\star,\eta^\star)$ is optimal in the dual problem \eqref{dual:problem:linearized} with $\lambda^\star = \lambda^\star(\eta^\star)$, then
	$$	
	    P^{\star}_{ij}=\frac{\lambda^\star\alpha_i P'_{ij}}{\eta^\star_i-C_{ij}}\quad \forall i,j\in\Xi.
	$$
\end{enumerate}
\end{myprop}
% The expression for $\lambda^\star(\eta)$ is derived from the first order condition together with the fact that $\sum_{i=1}^d \alpha_i=1, \sum_{j=1}^d P'_{ij}=1$ that 
% \begin{align}
% \frac{\partial}{\partial \lambda} \mathcal{J}(\eta, \lambda) &=r-1+\sum_{i=1}^{d} \sum_{j=1}^{d} \alpha_{i} P_{i j}^{\prime} \log \left(\frac{\lambda \alpha_{i}}{\eta_{i}-C_{i j}}\right)+\sum_{i=1}^{d} \sum_{j=1}^{d} \alpha_{i} P_{i j}^{\prime}=0 \\
% \implies &r-1+\sum_{i=1}^{d} \sum_{j=1}^{d} \alpha_{i} P_{i j}^{\prime} \log \left(\frac{\lambda \alpha_{i}}{\eta_{i}-C_{i j}}\right)+1=0\\
% &r+\sum_{i=1}^{d} \sum_{j=1}^{d} \alpha_{i} P_{i j}^{\prime} \log \left(\frac{\lambda^{\star} \alpha_{i}}{\eta_{i}-C_{i j}}\right)=0\label{eq:r:Q}
% \end{align}

\begin{proof}[Proof of Proposition~\ref{prop:linear:predictor}]
	Assume first that~$P'>0$. By dualizing the conditional relative entropy constraint as well as the row-wise normalization conditions for the stochastic matrix~$P$, the primal convex program \eqref{primal:problem:linearized} can be reformulated as
		\begin{align}
		\mc J^\star &=\max _{P \geqslant 0} \sum_{i,j=1}^d C_{ij} P_{ij} +\inf _{\lambda \geqslant 0} \lambda\left(r-\sum_{i=1}^{d} \alpha_{i} \D{P'_{i\cdot}}{P_{i\cdot}}\right)+\inf_{\eta \in \mathbb{R}^{d}} \eta^\top(\mathbf 1-P\mathbf 1) \nonumber \\ 
		&=\min _{\eta \in \mathbb{R}^{d}, \lambda \geqslant 0} \lambda r+\sum_{i=1}^{d} \eta_{i}+\sup _{P \geqslant 0}\left\{ \sum_{i,j=1}^d C_{ij} P_{ij} -\eta^\top P \mathbf{1} -\lambda \alpha_{i} \D{P'_{i\cdot}}{P_{i\cdot}}\right\} \nonumber \\
		&=\min _{\eta \in \mathbb{R}^{d}, \lambda \geqslant 0} \lambda r+\sum_{i=1}^{d} \eta_{i}+\sum_{i,j=1}^{d} \sup _{P_{ij} \geqslant 0}\left\{C_{ij} P_{ij}-\eta_{i} P_{ij}-\lambda \alpha_{i} P'_{ij} \log \frac{P'_{ij}}{P_{ij}}\right\},
		\label{eq:P-from-lambda-eta}
		\end{align}
		where the second equality follows from strong duality, which holds because $P=P'>0$ constitutes a Slater point for the primal convex program~\eqref{primal:problem:linearized}, whereas the third equality exploits the definition of the relative entropy.
		%from Sion's minimax theorem, which applies because  from the definition of ${\mc P}$ and the definition of $\Delta_d$. Due to the compactness of ${\mc P}$ and continuity and convexity of the relative entropy function in the second coordinate, it follows from Sion's minimax theorem \cite{ref:Sion-58} that we can interchange $\max$ with $\min$, which yields \eqref{1c}. By the separability of the cost function, \eqref{1d} follows directly. Writing out the expression for the relative entropy and applying separability again gives \eqref{1e}. 
		Denoting the Burg entropy by $\phi(t)= -\log t+t-1$, we then obtain
        \begin{align*}
		\mc J^\star &= \min _{\eta \in \mathbb{R}^{d}, \lambda \geqslant 0} \lambda r+\sum_{i=1}^{d} \eta_{i}+\sum_{i,j=1}^{d} \sup _{P_{ij} \geqslant 0}\left\{\left(C_{ij}-\eta_{i}\right) P_{ij}-\lambda \alpha_{i} \left(P'_{ij} \phi\left(\frac{P_{ij}}{P'_{ij}}\right)-P_{ij}+P'_{ij}\right)\right\}\\
		&= \min _{\eta \in \mathbb{R}^{d}, \lambda \geqslant 0} \lambda (r-1) +\sum_{i=1}^{d} \eta_{i}+ \sum_{i,j=1}^{d}  \sup _{P_{ij} \geqslant 0}\left\{\left(C_{ij}-\eta_{i}+\lambda \alpha_{i}\right) P_{ij}-\lambda \alpha_{i} P'_{ij} \phi\left(\frac{P_{ij}}{P'_{ij}}\right)\right\} \\
		%&= \min_{\eta \in \mathbb{R}^{d}, \lambda \geqslant 0}  \lambda\left(r-1\right)-\sum_{i=1}^{d} \eta_{i} +\sum_{i,j=1}^{d}  \sup_{t \geqslant 0} \left\{\left(C_{ij}+\eta_{i}+\lambda \alpha_{i}\right)t P'_{ij}-\lambda \alpha_{i} P'_{ij} \phi(t)\right\} \\
		&= \min _{\eta \in \mathbb{R}^{d}, \lambda \geqslant 0} \lambda\left(r-1\right)+\sum_{i=1}^{d} \eta_{i}+ \lambda \sum_{i,j=1}^{d} \alpha_{i} P'_{ij} \sup_{t \geqslant 0}\left\{\frac{\left(C_{ij}-\eta_{i}+\lambda \alpha_{i}\right)}{\lambda \alpha_{i}} t-\phi(t)\right\},
		\end{align*}
% 		where \eqref{2a} can be justified as the following: for 
% \begin{equation}\label{conjugate:optimizer}		
% t=\frac{P_{ij}}{P'_{ij}}, 
% \end{equation}
% then $P'_{ij}\log (P'_{ij}/P_{ij})= -P'_{ij}\log ({P_{ij}}/{P'_{ij}})=P'_{ij}(\phi({P_{ij}}/{P'_{ij}})-{P_{ij}}/{P'_{ij}}+1)=P'_{ij} \phi({P_{ij}}/{P'_{ij}})-P_{ij}+P'_{ij}$. 
where the second equality holds because $\sum_{i,j=1}^{d} \alpha_i P'_{ij} =1$, and the third equality follows from the substitution $t\leftarrow {P_{ij}}/{P'_{ij}}$. %The conjugate of a function $f: \R_+ \to \mb R$ is defined as $f^{*}(s)=\sup _{t\ge 0} \{st-f(t)\}$. Note that in \eqref{2d} we achieved an expression similar to the definition of the conjugate function $\phi^*(s)$ where $s={(C_{ij}+\tilde{\eta}_{i}+\lambda \alpha_{i})}/({\lambda \alpha_{i}})$. Now continuing with \eqref{2d}
We know from \citep[Table~4]{ben2013robust} that the conjugate of the Burg entropy is given by
\begin{equation*}
				\phi^*(s)=\sup_{t\geqslant0} \left\{st-\phi(s)\right\} %=\sup_{t\geqslant0} \left\{st-(-\log t +t -1 )\right\} 
				= \left\{\begin{array}{cl}
				-\log(1-s) & \text{if } s<1,\\
				\infty & \text{if }s\geqslant 1.
				\end{array}\right.
\end{equation*}
By identifying $s$ with $(C_{ij}-\eta_{i}+\lambda \alpha_{i})/(\lambda \alpha_{i})$, we then obtain~\eqref{dual:problem:linearized}. Thus, Assertion~\ref{duality:thm:i} follows.
% 		\begin{align*}
% 		\mc J^\star %&=\min _{\substack{\eta \in \mathbb{R}^{d},\lambda \geqslant 0\\\ - \eta_i> \max_j\{C_{ij}\}}} \lambda\left(r-1\right)-\sum_{i=1}^{d} \eta_{i}+\lambda \sum_{i,j=1}^{d} \alpha_{i} P'_{ij} \log \left(\frac{\lambda \alpha_{i}}{-C_{ij}-\eta_{i}}\right)\\
% 		&=\min_{\substack{\eta \in \mathbb{R}^{d},\lambda \geqslant 0\\\eta_i> \max_j\{C_{ij}\}}}  \lambda\left(r-1\right)+\sum_{i=1}^{d} \eta_{i}+\lambda \sum_{i, j=1}^{d} \alpha_{i} P'_{ij} \log \left(\frac{\lambda \alpha_{i}}{\eta_{i}-C_{ij}}\right), %\\
% 		%&=\min_{\lambda \geqslant 0} \lambda (r-1) + \sum_{i=1}^{d}\min_{\eta_i> \max_j\{C_{ij}\}}\eta_{i} +\lambda  \sum_{j=1}^{d} \alpha_{i} P'_{ij} \log \left(\frac{\lambda \alpha_{i}}{\eta_{i} -C_{ij}}\right) \\
% 	%&=\min_{\lambda \geqslant 0} \lambda (r-1) + \sum_{i=1}^{d}\min_{\eta_i\geqslant \max_j\{C_{ij}\}}\eta_{i} +\lambda  \sum_{j=1}^{d} \alpha_{i} P'_{ij} \log \left(\frac{\lambda \alpha_{i}}{-C_{ij}+\eta_{i}}\right),\label{3d}\\
% 	%&=\min_{\lambda \geqslant 0} \min_{\eta_i\geqslant \max_j\{C_{ij}\}} \mc J(\eta,\lambda)	\label{def:J:eta:lambda}
% 	\end{align*}
% 	which is manifestly equivalent to \eqref{dual:problem:linearized}. This observation completes the proof of Assertion~\ref{duality:thm:i}.
	
	As for Assertion~\ref{duality:thm:ii}, define
	$$
        \mathcal{J}(\lambda, \eta)=\lambda(r-1) + \sum_{i=1}^{d} \eta_{i}+\lambda  \sum_{j=1}^{d} \alpha_{i} P_{i j}^{\prime} \log \left(\frac{\lambda \alpha_{i}}{\eta_{i} -C_{i j}}\right)
    $$
    as the objective function of the dual problem~\eqref{dual:problem:linearized}, fix any~$\eta\in\R^d$ with~$\eta_i>\max_j \{C_{ij}\}$ for every~$i\in\Xi$, and use~$\lambda^\star(\eta)$ to denote the unique minimizer of the parametric convex optimization problem $\min_{\lambda\geqslant 0} \mathcal J(\lambda, \eta)$. Next, note that $\lambda^\star(\eta)$ must satisfy the first-order optimality condition
			\begin{align*}
			0=\frac{\partial}{\partial \lambda}\mc J(\lambda,\eta) = r-1+\sum_{i,j=1}^{d} \alpha_{i} P'_{ij} \log \left(\frac{\lambda \alpha_{i}}{\eta_{i}-C_{ij}}\right) +\sum_{i,j=1}^{d}\alpha_{i} P'_{ij} =r+\sum_{i,j=1}^{d} \alpha_{i} P'_{ij} \log \left(\frac{\lambda \alpha_{i}}{\eta_{i}-C_{ij}}\right),
		\end{align*}
	where the last equality holds again because $\sum_{i,j=1}^{d} \alpha_i P'_{ij} =1$. Solving the above equation for~$\lambda$ yields
% 	Applying the first order optimality condition gives
% 	\begin{subequations}
% 	\begin{align}
% 	r+\sum_{i=1}^{d}\sum_{j=1}^{d} \alpha_{i} P'_{ij} \log \left(\frac{\lambda^{\star} \alpha_{i}}{\eta_{i}-C_{ij}}\right)&=0 \label{eq:SGD} \\
% 	\iff r+\sum_{i=1}^{d}\alpha_{i} \log (\lambda^{\star})&=\sum_{i=1}^{d}\sum_{j=1}^{d} \alpha_{i} P'_{ij} \log \left(\frac{\eta_{i}-C_{ij}}{\alpha_{i}}\right)	\\
% 	\log (\lambda^{\star})&= \sum_{i=1}^{d}\sum_{j=1}^{d} \alpha_{i} P'_{ij} \log \left(\frac{\eta_{i}-C_{ij}}{\alpha_{i}}\right)-r,
% 	\end{align}
% 	\end{subequations}
% 	and thus the optimal $\lambda^\star$ is given by
	\begin{equation*}
		\lambda^\star (\eta)=\exp \left(\sum_{i,j=1}^{d} \alpha_{i} P'_{ij} \log \left(\frac{\eta_i-C_{ij}}{\alpha_{i}}\right)-r\right),
		\end{equation*}
and hence Assertion \ref{duality:thm:ii} follows.

Next, we prove Assertion \ref{duality:thm:iv}. To this end, select any dual optimal~$\eta^\star$, and substitute~$\lambda^\star(\eta^\star)$ into~$\mathcal J(\lambda,\eta^\star)$ to obtain
	\begin{align*}
	\mc J^\star=\mc J(\lambda^\star(\eta^\star), \eta^\star) &=
        \sum_{i=1}^d \eta^\star_i-\exp \left(\sum_{i,j=1}^{d} \alpha_{i} P'_{ij} \log \left(\frac{\eta^\star_i-C_{ij}}{\alpha_{i}}\right)-r\right) ,
        %&\geqslant \sum_{i=1}^d \eta_i-e^{-r}\left(\sum_{i,j=1}^{d} P'_{ij} (\eta_i-C_{ij})\right)= \sum_{i=1}^d \eta_i(1- e^{-r})+e^{-r}\left(\sum_{i,j=1}^{d} P'_{ij}C_{ij}\right),
    \end{align*}
    %On the one hand, this formula implies that $\mc J(\lambda^\star(\eta),\eta)\leqslant \sum_{i=1}^d \eta_i$.
    %If $\eta$ is feasible in~\eqref{dual:problem:linearized}, then $\eta_i> \max_j\{C_{ij}\}$ for all~$i\in\Xi$, and thus the above formula implies that $\mc J(\lambda^\star(\eta),\eta)\leqslant \sum_{i=1}^d \eta_i < \sum_{i=1}^d \max_j\{C_{ij}\}$. 
    %Other other hand, we may use 
    where the first equality follows from Assertions~\ref{duality:thm:i} and~\ref{duality:thm:ii}. As $\sum_{i,j=1}^{d} \alpha_i P'_{ij} =1$, we may then use Jensen's inequality to interchange the sum over $i$ and $j$ with the logarithm to obtain
\begin{align*}
	%\mc J(\lambda^\star(\eta), \eta)
	\mc J^\star&\geqslant \sum_{i=1}^d \eta^\star_i-e^{-r}\left(\sum_{i,j=1}^{d} P'_{ij} (\eta^\star_i-C_{ij})\right)= \sum_{i=1}^d \eta^\star_i(1- e^{-r})+e^{-r}\left(\sum_{i,j=1}^{d} P'_{ij}C_{ij}\right).
\end{align*}
By inspection of the primal objective function, we further have the trivial upper bound $\mathcal{J}^\star \leq d \max_{i,j}C_{ij}$. Combining these upper and lower bounds on $\mc J^\star$ yields
\begin{align*}
	\sum_{i=1}^d \eta^\star_i(1- e^{-r})+e^{-r}\left(\sum_{i,j=1}^{d} P'_{ij}C_{ij}\right)\leqslant d \max_{i,j}C_{ij}.
\end{align*}
% If $\eta$ is feasible in~\eqref{dual:problem:linearized}, then $\eta_i> \max_j\{C_{ij}\}$ for all~$i\in\Xi$, {\color{red} OLD: and thus
% \begin{equation*}
% 	\eta_i^\star \leqslant (1- e^{-r})^{-1} \sum_{i=1}^d\left(\max_j\{C_{ij}\}-e^{-r}\sum_{j=1}^{d} P'_{ij} C_{ij}\right) - \sum_{\substack{k=1\\k\ne i}}^d \max_j\{C_{kj}\},
% \end{equation*} }
As~$\eta^\star$ must satisfy the dual constraints $\eta^\star_i> \max_j\{C_{ij}\}$ for all~$i\in\Xi$, we may finally conclude that
% For any $\eta, \lambda$ feasible in~\eqref{dual:problem:linearized}, we have $\mathcal{J}(\lambda,\eta) \leq d \max_{i,j}C_{ij}$. Combining this upper bound with \eqref{eq:pf:duality:side:res:1} gives
% \begin{equation*}
%   d \max_{i,j}C_{ij} 
%   \geq \sum_{i=1}^d \eta_i(1- e^{-r})+e^{-r}\left(\sum_{i=1}^{d}\sum_{j=1}^{d} P'_{ij}C_{ij}\right).
% \end{equation*}
% Using then $\eta_i> \max_j\{C_{ij}\}$ results in
\begin{equation*}
    \eta^\star_i \leq \frac{d \max_{i,j}C_{ij}}{1-e^{-r}} - \sum_{\substack{k=1\\k\ne i}}^d \max_j\{C_{kj} \} - \frac{e^{-r}}{1-e^{-r}}\left(\sum_{i,j=1}^{d} P'_{ij}C_{ij}\right) \quad \forall i\in\Xi.
\end{equation*}
This observation completes the proof of Assertion~\ref{duality:thm:iv}.

As for Assertion~\ref{duality:thm:iii}, note that the primal maximizer $P^\star$ can be computed cheaply from the dual minimizer~$(\lambda^\star,\eta^\star)$ by solving the inner maximization problem in~\eqref{eq:P-from-lambda-eta} for~$\lambda=\lambda^\star$ and~$\eta=\eta^\star$. Indeed, $P^\star$ must satisfy the first-order condition
\[
    C_{ij}-\eta^\star_i+\frac{\lambda^\star \alpha_i P'_{ij}}{P^\star_{ij}} = 0\quad\implies\quad P^{\star}_{ij}=\frac{\lambda^\star\alpha_i P'_{ij}}{\eta_i^\star-C_{ij}}
\]
for all~$i,j\in\Xi$, and thus the claim follows. We remark that if~$P'\not> 0$, then the proofs of Assertions~\ref{duality:thm:i}-\ref{duality:thm:iii} become more technical and require tedious case distinctions. However, no new ideas are needed. Details are omitted for brevity.
% in the proof above, i.e., $P^\star=t^\star P'$ as the following. Writing out the definition of the conjugate function, we have 
% 	\begin{equation*}
% 				\phi^*(s)=\sup_{t\geqslant0} \left\{st-\phi(s)\right\}=\sup_{t\geqslant0} \left\{st-(-\log t +t -1 )\right\},	
% 	\end{equation*}
% 	where $t^\star$ can be derived from the first order condition
% 	\begin{equation}\label{eq:optimal:t}
% 		\frac{\partial}{\partial t}\left(st-(-\log t +t -1 )\right)=s+\frac{1}{t}-1=0\implies t^\star=\frac{1}{1-s}
% 	\end{equation}
% 	where $s<1$. In our case $s=(C_{ij}+\tilde{\eta}_{i}+\lambda \alpha_{i})/(\lambda \alpha_{i})$ so that according to \eqref{eq:optimal:t} $t^\star=\lambda^\star\alpha_i/(\eta_i-C_{ij})$, which implies 
% 	\begin{equation}\label{proof:primal:optimizer}
% 				P^{\star}_{ij}=\frac{\lambda^\star\alpha_i P'_{ij}}{\eta_i^\star-C_{ij}}.
% 	\end{equation}
\end{proof}

In the following we denote by $Q(\eta)=\mc J(\lambda^\star(\eta), \eta)$ the partial minimum of $\mc J(\lambda,\eta)$ with respect to~$\lambda\ge 0$. From the proof of Proposition~\ref{prop:linear:predictor}\,\ref{duality:thm:iv} we know that~$Q(\eta)=\sum_{i=1}^d Q_i(\eta)$, where~$Q_i(\eta)=\eta_{i}- \lambda^\star(\eta)/d$ for all~$i\in\Xi$. By Proposition~\ref{prop:linear:predictor}\,\ref{duality:thm:ii}, the dual oracle subproblem~\eqref{dual:problem:linearized} is therefore equivalent to
\begin{equation}
    \label{eq:dual-partial-minimum}
    \mc J^\star = \min_{\eta\in[\underline\eta,\overline\eta]} \sum_{i=1}^d Q_i(\eta),
\end{equation}
where the variable bounds $\underline\eta,\overline\eta\in\R^d$ are defined through
\begin{align*}
\underline\eta_i = \max_j \{C_{ij}\}\quad\text{and}\quad
    \bar{\eta}_{i}=\frac{1}{1-e^{-r}}\left(d \max_{i,j}C_{ij}-e^{-r} \tr {C^\top P'} \right)-\sum_{k \neq i} \underline{\eta}_{k} \quad \forall i\in\Xi.
\end{align*}
Problem~\eqref{eq:dual-partial-minimum} is amenable to stochastic gradient descent algorithms.

\begin{proof}[Proof of Theorem~\ref{thm:correct:algsub}]
	The claim follows directly from the convergence results in~\citep[Section~2.2]{nemirovski09Robust} applied to the reformulation~\eqref{eq:dual-partial-minimum} of the dual oracle subproblem~\eqref{dual:problem:linearized}.
\end{proof}

%%%%%%%%%%%%%%%%%%%%%%%%
\section{Proofs for Section~\ref{sect:hypo:test}}

\begin{proof}[Proof of Proposition~\ref{eq:prop:HT}]

By the large deviation principle for the estimator $\widehat \theta_T$ established in Lemma~\ref{Lem:LDP}, we have
\begin{align*} 
    \limsup_{T\to\infty}\frac{1}{T} \log \alpha_T \leq -\inf_{\theta'\not \in\interior \mathcal{B}} \Dc{\theta'}{\theta^{(1)}}
    =-\inf_{\theta'\not\in \mathcal{B}} \Dc{\theta'}{\theta^{(1)}}
    = -r,
\end{align*}
where the first equality holds because $\mathcal{B}$ is open thanks to the continuity of $\Dc{\theta'}{\theta}$; see Proposition~\ref{prop:properties:conditional:entropy}. To prove the second equality, denote by $\theta^\star$ a minimizer of $\min_{\theta'\in\Theta'} \{\Dc{\theta'}{\theta^{(1)}} : \Dc{\theta'}{\theta^{(1)}}\geq \Dc{\theta'}{\theta^{(2)}}\}$, which exists due to the compactness of~$\Theta'$ and the continuity of~$\Dc{\theta'}{\theta}$ on~$\Theta'\times\Theta$. It is then clear that $\Dc{\theta^\star}{\theta^{(1)}}= \Dc{\theta^\star}{\theta^{(2)}}=r$.

Similarly, the large deviation principle for $\widehat \theta_T$ implies that
\begin{align*}
    \liminf_{T\to\infty}\frac{1}{T} \log \alpha_T \geq -\inf_{\theta'\not\in\cl \mathcal{B}} \Dc{\theta'}{\theta^{(1)}}
    =-\inf_{\theta'\not\in \mathcal{B}} \Dc{\theta'}{\theta^{(1)}}
    =-r,
\end{align*}
where the first equality follows again from the continuity of $\Dc{\theta'}{\theta}$. In summary, we have thus shown that $\lim_{T\to\infty}\frac{1}{T} \log \alpha_T = -r$.
 \begin{figure}[htb] 
 \centering
   \scalebox{0.7}{\begin{tikzpicture}
\begin{axis}[
anchor=center,width=2.5in,
height=1.5in,
xmax=1,xmin=0,
ymin=0,ymax=1,
xticklabels={},
yticklabel=\ ,
	extra description/.code={
		% this generates custom y labels to implement 
		% individual styles for every tick:
		\node[below] at (axis cs:0.4,-0.02) {$\theta^\star$};
		\node[below] at (axis cs:0.8,-0.01) {$\xcancel{\theta^\star}$};
		\node[below] at (axis cs:0.5,0) {$\theta^{(1)}$};
		\node[below] at (axis cs:0.2,0) {$\theta^{(2)}$};
	},
at={(1.011111in,0.641667in)},
scale only axis,
legend style={legend cell align=left,align=left,draw=white!15!black}
]
\addplot [color=blue,solid, smooth, line width=0.8]{(x-0.2)^2} ;
\addlegendentry{$\Dc{\bar\theta}{\theta^{(2)}}$};

  \addplot [color=red,solid, smooth,line width=0.8]{4*(x-0.5)^2} ;
  \addlegendentry{$\Dc{\bar\theta}{\theta^{(1)}}$};
% \addplot[mark=square,solid] coordinates {(0.4, 0.05)};
\addplot[mark=*,color=green] coordinates {(0.4,0.04)} ;
\addplot[mark=*] coordinates {(0.8,0.36)} ;
\node[diamond,fill,color=blue,inner sep=2pt] at (axis cs:0.2,0) {};
\node[diamond,fill,color=red,inner sep=2pt] at (axis cs:0.48,0) {};
% \addplot[mark=diamond,fill,color=black] coordinates {(0.2,0)} ;
\addplot[mark=none,dashed] coordinates {(0.8,0.36) (0.8,0)};
\addplot[mark=none,solid] coordinates {(0.4,0.04) (0.4,0)} ;

%   \addlegendentry{}
\end{axis}
\end{tikzpicture}}
     \caption[]{Illustration of $\theta^{\star}$ relative to $\theta^{(1)}$ and $\theta^{(2)}$.}
     \label{fig:illustration:theta}
 \end{figure}
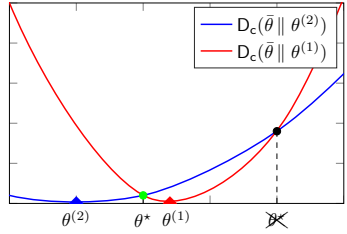
The proof of the assertion $\lim_{T\to\infty}\frac{1}{T} \log \beta_T = -r$ is analogous and thus omitted for brevity.
\end{proof}

\begin{proof}[Proof of Theorem~\ref{thm:optimality:HT}]
Assume for the sake of argument that there exists a decision rule induced by some open set $\bar{\mathcal{B}}$ such that the corresponding error probabilities satisfy $\lim_{T\to\infty}\frac{1}{T}\log\bar\beta_T < -r$ and $\lim_{T\to\infty}\frac{1}{T}\log\bar\alpha_T \leq -r$. %We will show that this leads to a contradiction. 
The first assumption implies via the large deviation principle for $\widehat \theta_T$ that
\begin{equation}\label{eq:pf:HT:contr:ass}
   \inf_{\theta'\in \cl\bar{\mathcal{B}}} \Dc{\theta'}{\theta^{(2)}}= -\lim_{T\to\infty}\frac{1}{T}\log\bar\beta_T > r,
\end{equation}
where the equality follows from the continuity of~$\Dc{\theta'}{\theta}$ and the assumption that $\bar{\mathcal{B}}$ is open. Next, define~$\theta^\star$ as in the proof of Proposition~\ref{eq:prop:HT}, and recall that $\Dc{\theta^\star}{\theta^{(1)}}=\Dc{\theta^\star}{\theta^{(2)}}=r$. The inequality~\eqref{eq:pf:HT:contr:ass} thus implies that~$\theta^\star \notin \cl\bar{\mathcal{B}}$. % and accordingly $\theta^\star\in\interior \bar{\mathcal{B}}^\mathsf{c}$. 
In addition, choose any $\varepsilon>0$ with $\theta^\star_\varepsilon=(1-\varepsilon)\theta^\star + \varepsilon \theta^{(1)}\in\Theta \not\in \cl\bar{\mc B}$, which exists because~$\cl\bar{\mc B}$ has an open complement. As it is easy to verify that~$\Dc{\theta'}{\theta}$ is strictly convex in~$\theta'$ for any~$\theta\in\Theta$, we have
$\Dc{\theta^\star_\varepsilon}{\theta^{(1)}} < \Dc{\theta^\star}{\theta^{(1)}}=r$ and
\begin{equation*}
    \inf_{\theta'\not\in \cl \bar{\mathcal{B}}} \Dc{\theta'}{\theta^{(1)}}\leq \Dc{\theta^\star_\varepsilon}{\theta^{(1)}} < r.
\end{equation*}
This in turn implies via the large deviation principle for $\widehat \theta_T$ that
\begin{align*}
 \liminf_{T\to\infty}\frac{1}{T}\log \bar\alpha_T =
    \liminf_{T\to\infty}\frac{1}{T}\log \mathbb{P}_{\theta^{(1)}}(\widehat\theta_T\not\in \bar{\mathcal{B}}) 
    \geq -\inf_{\theta'\not\in\cl\bar{\mathcal{B}}} \Dc{\theta'}{\theta^{(1)}} > -r,
\end{align*}
which contradicts our second assumption. Thus, the claim follows.
\end{proof}

\subsection{Coin Tossing with Markovian Coins}\label{section:MCoin:example}
We now highlight that the decay rate $r$ of the error probabilities in Proposition~\ref{eq:prop:HT} can be arbitrarily small, which indicates that it can be arbitrarily difficult to distinguish two Markov chains based on finitely many samples.
To this end, we consider a coin flipping example involving a hypothetical Markovian coin, where the probability of seeing a head (H) or tail (T) depends on the outcome of the last coin toss. That is, we assume that $\{\xi_t\}_{t\in\N}$ constitutes a stationary Markov chain with finite state space $\Xi = \{{\rm H},{\rm T}\}$, and we denote by 
\[
	\Theta=\{\theta\in\Re^{2\times 2}_{++}:\theta_{\rm HH}+2\theta_{\rm HT}+\theta_{\rm TT}=1,~\theta_{\rm HT}=\theta_{\rm TH}\}
\]
the set of all possible balanced probability mass functions of the doublet $(\xi_t,\xi_{t+1})$. We also define the estimator $\widehat \theta_T$ as the empirical doublet distribution after $T$ tosses as defined in \eqref{MC:estimator}. In the following, we consider two different Markov coins with doublet distributions $\theta^{(1)}$, $\theta^{(2)}\in\Theta$. Specifically, we assume henceforth that the doublet distribution $\theta^{(1)}$ and the corresponding transition probability matrix $P_{\theta^{(1)}}$ of the first coin are given by
\[
	\theta^{(1)}= \begin{pmatrix} \frac{1-\varepsilon}{2} & \frac{\varepsilon}{2} \\ \frac{\varepsilon}{2} & \frac{1-\varepsilon}{2} \end{pmatrix} \quad \mbox{and} \quad P_{\theta^{(1)}} = \begin{pmatrix} 1-\varepsilon & \varepsilon \\ \varepsilon & 1-\varepsilon \end{pmatrix},
\]
respectively, for some $\varepsilon \in(0,1)$. The unique stationary distribution corresponding to this model is given by $\pi_{\theta^{(1)}} = (1/2, 1/2)$. We compare the reference coin encoded by $\theta^{(1)}$ with a second Markov coin parameterized in $\varepsilon$, whose doublet distribution and transition probability matrix are given by
\[
	\theta^{(2)}= \begin{pmatrix} \varepsilon^2 & \varepsilon(1-\varepsilon) \\ \varepsilon(1-\varepsilon) & (1-\varepsilon)^2 \end{pmatrix}\quad \mbox{and} \quad P_{\theta^{(2)}} = \begin{pmatrix} \varepsilon & 1-\varepsilon \\ \varepsilon & 1-\varepsilon \end{pmatrix},
\]
respectively. The stationary distribution corresonding to $\theta^{(2)}$ is given by $\pi_{\theta^{(2)}}= (\varepsilon, 1-\varepsilon)$. Note that the second coin is actually memoryless and follows an i.i.d.\ process.
\begin{figure}
\centering
\subfigure[Markov coin corresponding to $\theta^{(1)}$]{% !TEX root = ../vod.tex

\begin{tikzpicture}[->,>=stealth',shorten >=1pt,auto,node distance=2.8cm]
\tikzstyle{every state}=[draw=black,thick,text=black,scale=1]
\node[state]         (H)              {H};
\node[state]         (T) [right of=H] {T};
\path (T) edge  [bend right] node[above] {$\varepsilon$} (H);
\path (H) edge  [bend right] node[below] {$\varepsilon$} (T);
\path (H) edge  [loop left] node {$1-\varepsilon$} (H);
\path (T) edge  [loop right] node {$1-\varepsilon$} (T);
\end{tikzpicture}\label{fig:markov_Q}} \qquad
\subfigure[Markov coin corresponding to $\theta^{(2)}$]{% !TEX root = ../vod.tex

\begin{tikzpicture}[->,>=stealth',shorten >=1pt,auto,node distance=2.8cm]
\tikzstyle{every state}=[draw=black,thick,text=black,scale=1]
\node[state]         (H)              {H};
\node[state]         (T) [right of=H] {T};
\path (T) edge  [bend right] node[above] {$\varepsilon$} (H);
\path (H) edge  [bend right] node[below] {$1-\varepsilon$} (T);
\path (H) edge  [loop left] node {$\varepsilon$} (H);
\path (T) edge  [loop right] node {$1-\varepsilon$} (T);
\end{tikzpicture}\label{fig:markov_P}}
\caption{Transition probability diagrams of the Markov coins corresponding to the models $\theta^{(1)}$ and $\theta^{(2)}$. Although the Markov coins are vastly different for small $\varepsilon$, it is nevertheless hard to distinguish them based on finitely many samples.}
\label{fig:markov_coins}
\end{figure}
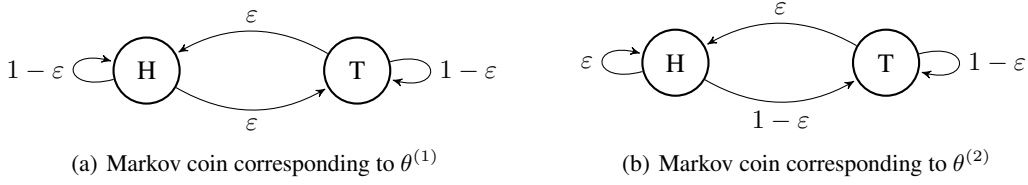
The conditional relative entropy of $\theta^{(2)}$ with respect to the reference coin $\theta^{(1)}$ evaluates to
\begin{align*}
	\Dc{\theta^{(2)}}{\theta^{(1)}} ~&=\sum_{i\in\Xi} (\pi_{\theta^{(2)}})_i \, \mathsf{D}((P_{\theta^{(2)}})_{i\cdot}\|(P_\theta^{(1)})_{i\cdot})  
%	& = \varepsilon \D{(\varepsilon, 1-\varepsilon)}{(1-\varepsilon, \varepsilon)} + (1-\varepsilon) \D{(\varepsilon, 1-\varepsilon)}{(\varepsilon, 1-\varepsilon)}\nonumber \\
%	& = \varepsilon \left(\varepsilon \log\left( \frac{\varepsilon}{1-\varepsilon}\right) +(1-\varepsilon) \log\left( \frac{1-\varepsilon}{\varepsilon}\right) \right) 
= \varepsilon (1-2\varepsilon) \log\left(\frac{1-\varepsilon}{\varepsilon}\right), 
\end{align*}
which is visualized in Figure~\ref{fig:markovian:coin}. Observe that $\Dc{\theta^{(2)}}{\theta^{(1)}}$ vanishes for $\varepsilon=1/2$, in which case the two Markov coins are stochastically indistinguishable. Maybe surprisingly, however, $\Dc{\theta^{(2)}}{\theta^{(1)}}$ also vanishes in the limit $\varepsilon\downarrow 0$, in which case the two Markov coins are dramatically dissimilar. This can be seen, for instance, by comparing the limiting stationary probability mass functions $\pi_{\theta^{(1)}}=(1/2, 1/2)$ and $\pi_{\theta^{(2)}}=(0, 1)$. If $\varepsilon$ is small, Lemma~\ref{Lem:LDP} implies that for any Borel set $\mc D\subseteq \Theta$ with $\theta^{(2)}\in\mc D$ and $\theta^{(1)}\notin \mc D$, the probability $\mb P_{\theta^{(1)}}(\hat \theta_T\in \mc D)$ decays slowly as the sample size $T$ increases. Thus, it is hard to distinguish the coins corresponding to $\theta^{(1)}$ and $\theta^{(2)}$ solely based on data in spite of their fundamental differences. Also the error probability decay rates of the decision rule in Proposition~\ref{eq:prop:HT} are arbitrarily small in this setting.
%%%%%%%% 
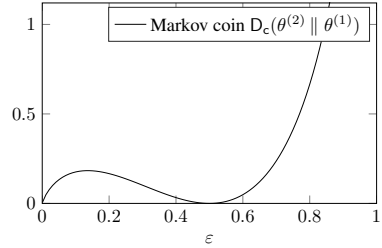
\begin{figure}[htb] 
\centering
   \scalebox{0.7}{ % This file was created by matlab2tikz.
% Minimal pgfplots version: 1.3
%  http://www.mathworks.com/matlabcentral/fileexchange/22022-matlab2tikz
\definecolor{mycolor1}{rgb}{0.00000,0.44700,0.74100}%
\definecolor{mycolor2}{rgb}{0.85000,0.32500,0.09800}%
\begin{tikzpicture}

\begin{axis}[%
width=2.5in,
height=1.5in,
at={(1.011111in,0.641667in)},
scale only axis,
xmin=0,
xmax=1,
ymin=0,
ymax=1.12007299270073,
xlabel={$\varepsilon$},
legend style={legend cell align=left,align=left,draw=white!15!black}
]
\addplot [color=black,solid, line width=0.5]
  table[row sep=crcr]{%
0.001	0.00689294126909126\\
0.010989898989899	0.0483646206391645\\
0.020979797979798	0.0772422012671515\\
0.030969696969697	0.100032479989821\\
0.040959595959596	0.118579154383337\\
0.050949494949495	0.133824484267017\\
0.0609393939393939	0.14635580745603\\
0.0709292929292929	0.156581374795828\\
0.0809191919191919	0.164805545597715\\
0.0909090909090909	0.171266659809475\\
0.10089898989899	0.17615830529937\\
0.110888888888889	0.17964223460569\\
0.120878787878788	0.181856690759964\\
0.130868686868687	0.182922031200393\\
0.140858585858586	0.182944673420554\\
0.150848484848485	0.182019951481954\\
0.160838383838384	0.180234239550084\\
0.170828282828283	0.177666566794956\\
0.180818181818182	0.174389869971673\\
0.190808080808081	0.170471982007062\\
0.20079797979798	0.165976424411008\\
0.210787878787879	0.160963051371308\\
0.220777777777778	0.155488579997362\\
0.230767676767677	0.149607031987096\\
0.240757575757576	0.143370105557287\\
0.250747474747475	0.136827491891654\\
0.260737373737374	0.13002714703968\\
0.270727272727273	0.123015527758117\\
0.280717171717172	0.11583779796957\\
0.290707070707071	0.10853801114287\\
0.30069696969697	0.101159272856808\\
0.310686868686869	0.093743887006686\\
0.320676767676768	0.0863334884911506\\
0.330666666666667	0.0789691647310166\\
0.340656565656566	0.0716915679901533\\
0.350646464646465	0.0645410201673989\\
0.360636363636364	0.0575576114902993\\
0.370626262626263	0.050781294353073\\
0.380616161616162	0.0442519733926589\\
0.390606060606061	0.0380095927805013\\
0.40059595959596	0.0320942216181536\\
0.410585858585859	0.0265461382575712\\
0.420575757575758	0.0214059143188477\\
0.430565656565657	0.016714499146806\\
0.440555555555556	0.0125133054315582\\
0.450545454545455	0.00884429671583964\\
0.460535353535354	0.00575007752295911\\
0.470525252525253	0.00327398686344144\\
0.480515151515152	0.00146019591612586\\
0.49050505050505	0.000353810731284453\\
0.50049494949495	9.80870332618618e-07\\
0.510484848484849	0.000449014980482896\\
0.520474747474747	0.00174650440529549\\
0.530464646464647	0.00394345605618511\\
0.540454545454545	0.00709143591890492\\
0.550444444444444	0.0112437247470622\\
0.560434343434343	0.0164554877070517\\
0.570424242424242	0.022783959991929\\
0.580414141414141	0.0302886507237536\\
0.59040404040404	0.0390315678248584\\
0.600393939393939	0.0490774669708857\\
0.610383838383838	0.0604941282579704\\
0.620373737373737	0.0733526648429122\\
0.630363636363636	0.0877278685735298\\
0.640353535353535	0.103698598548393\\
0.650343434343434	0.12134821967136\\
0.660333333333333	0.14076509964902\\
0.670323232323232	0.162043174585861\\
0.680313131313131	0.185282595450871\\
0.69030303030303	0.210590470336115\\
0.700292929292929	0.238081720755855\\
0.710282828282828	0.267880074448597\\
0.720272727272727	0.300119222519549\\
0.730262626262626	0.334944175671908\\
0.740252525252525	0.372512863236609\\
0.750242424242424	0.412998030435395\\
0.760232323232323	0.456589504803741\\
0.770222222222222	0.503496923382856\\
0.780212121212121	0.553953040215389\\
0.79020202020202	0.608217771845231\\
0.800191919191919	0.666583191375177\\
0.810181818181818	0.729379755886807\\
0.820171717171717	0.796984157984139\\
0.830161616161616	0.869829346045468\\
0.840151515151515	0.948417485346674\\
0.850141414141414	1.03333697605129\\
0.860131313131313	1.12528517580504\\
0.870121212121212	1.22509931883314\\
0.880111111111111	1.33379950378764\\
0.89010101010101	1.45264995700211\\
0.900090909090909	1.58324888161075\\
0.910080808080808	1.72766474983094\\
0.920070707070707	1.8886515318945\\
0.930060606060606	2.07000560875707\\
0.940050505050505	2.27719472997555\\
0.950040404040404	2.51855609257906\\
0.960030303030303	2.80782794480736\\
0.970020202020202	3.17034227764064\\
0.980010101010101	3.66202430996714\\
0.99	4.45818527860058\\
};
\addlegendentry{Markov coin $\Dc{\theta^{(2)}}{\theta^{(1)}}$};

\end{axis}
\end{tikzpicture}%
}
    \caption[]{%Conditional relative entropy $\Dc{\theta^{(2)}}{\theta^{(1)}}$ between $\theta^{(1)}$ and $\theta^{(2)}$. 
    For $\varepsilon=1/2$ the Markov coins induced by the models $\theta^{(1)}$ and $\theta^{(2)}$ are indistinguishable, and thus $\Dc{\theta^{(2)}}{\theta^{(1)}}$ vanishes. However, $\Dc{\theta^{(2)}}{\theta^{(1)}}$ vanishes also as $\varepsilon$ approaches $0$ --- even though the Markov coins are dramatically different in this limit.}
    \label{fig:markovian:coin} 
\end{figure}

% \newpage

\section{DFO Algorithm for the Prescriptor Problem}
\label{sec:dfo-algorithm}

\begin{algorithm}[H]
	\SetKwInOut{Input}{Input}
	\SetKwInOut{Output}{Output}
	\Input{ $ N, \ x_0, \ \alpha_0, \ 0<\beta_{1} \leq \beta_{2}<1, \text { and } \gamma \geq 1$ }
	\Output{ $ \widehat{c}_{r}\left(x_N,\theta^{\prime}\right) \approx\min_{x\in X}\widehat{c}_{r}\left(x\right) $}
	\SetAlgoLined
	\DontPrintSemicolon
	\For{$k=0,\ldots , N$}{
		\textbf{Search step}: fetch $x $ with\\ $\widehat{c}_{r}\left(x,\theta^{\prime}\right)<\widehat{c}_{r}\left(x_k,\theta^{\prime}\right)-\alpha_{k}^2$; \\
		\eIf{found}{declare the search successful and set $ x_{k+1}=x $}{
		\textbf{Poll step}: Choose positive spanning set $D_k\subseteq \R^n$, and evaluate $ \widehat{c}_{r} $ at all poll points $P_k=\left\{x_{k}+\alpha_{k} d: d \in D_{k}\right\}.$ \\
		\eIf{$\widehat{c}_{r}\left(x_{k}+\alpha_{k} d_{k},\theta^{\prime}\right)<\widehat{c}_{r}\left(x_k,\theta^{\prime}\right)-\alpha_{k}^2$}{declare the iteration successful and set \\$ x_{k+1}=x_{k}+\alpha_{k} d_{k}. $}{declare the iteration unsuccessful and set \\ $ x_{k+1}=x_{k}$.}	
	}
	\eIf{the iteration was successful}{set $ \alpha_{k+1}\in[\alpha_k,\gamma\alpha_k] $}{set $ \alpha_{k+1}\in[\beta_{1}\alpha_k,\beta_{2}\alpha_k] $}}
	\caption{Directional direct-search method for the prescriptor problem~\eqref{def:DRO:decision}~\cite{vicente2013worst}}\label{alg:prescri}
\end{algorithm}

\end{document}